\documentclass[11pt]{article}
\usepackage[T1]{fontenc}
\usepackage{lmodern}
\usepackage[margin=1in]{geometry}
\usepackage{amsmath,amssymb,amsthm,mathtools,mathrsfs}
\usepackage{booktabs,array,longtable}
\usepackage[hidelinks]{hyperref}
\usepackage{microtype}
\allowdisplaybreaks
\numberwithin{equation}{section}
\newcommand{\ii}{\mathrm i}
\newcommand{\e}{\mathrm e}
\newcommand{\Rea}{\operatorname{Re}}
\newcommand{\Ima}{\operatorname{Im}}
\newcommand{\Li}{\operatorname{Li}}
\newcommand{\sgn}{\operatorname{sgn}}
\newtheorem{theorem}{Theorem}[section]
\newtheorem{lemma}[theorem]{Lemma}
\newtheorem{proposition}[theorem]{Proposition}

\theoremstyle{remark}

\hypersetup{pdftitle={A proof of the Berkovich--Dhar conjecture modulo five},pdfauthor={Shutao Jiang}}
\title{A proof of the Berkovich--Dhar conjecture modulo five}
\author{Shutao Jiang\\[3pt]
{\small East China Normal University}\\
{\small\href{mailto:51290147024@stu.ecnu.edu.cn}{\texttt{51290147024@stu.ecnu.edu.cn}}}}
\date{}
\begin{document}
\maketitle
\begin{abstract}

We prove the sign-pattern and limiting-transition assertions of the Berkovich--Dhar conjecture modulo five, covering both the square and the cube of the finite Borwein product. In each power, the coefficients in residue class zero are strictly positive, while the sequences in residue classes three and four have exactly one positive-to-negative sign change after zero terms are omitted. We determine all four transition constants, correct the preliminary numerical estimates in the conjecture, and give their linear corrections with explicit bounded errors. A common four-arc expansion governs both powers. Near each vanishing amplitude, strict negativity of a weighted adjacent difference determines the transition. The cubic case also has two identically vanishing components in the infinite product; a shifted saddle expansion resolves the resulting finite-product boundary terms. Positive series identities, uniform remainder estimates, and an exact integer recurrence complete the proof for every positive integer.

\end{abstract}


\section{Introduction and main results}

Sign patterns of finite products connect partition identities with asymptotic analysis. The Borwein conjectures, recorded by Andrews \cite{Andrews}, provide classical examples. Wang and Krattenthaler \cite{WK} developed a direct saddle point approach to such problems. Motivated by this work, Berkovich and Dhar \cite{BD} proposed further sign patterns for powers of finite products modulo three and five. The present paper resolves the structural assertions and limiting-transition assertions in their Conjecture 2.2.

For integers \(n\ge1\) and \(p\in\{2,3\}\), write

\begin{equation}\label{eq:product}
 Q_n(q)=\prod_{k=0}^{n-1}\prod_{a=1}^4(1-q^{5k+a}),\qquad
 Q_n(q)^p=\sum_{d=0}^{10pn^2}c_d^{(p)}(n)q^d.
\end{equation}

The coefficients satisfy

\begin{equation}\label{eq:reciprocity}
 c_d^{(p)}(n)=c_{10pn^2-d}^{(p)}(n).
\end{equation}

Thus the sequences in residue classes one and two are the reversals of those in classes four and three. Nonnegativity in residue class zero for both powers was already conjectured in \cite[Remark 11.6(2)]{WK}; see also \cite[Remark 3]{BD}. Our first theorem establishes strict positivity and the two single-transition assertions simultaneously.

\begin{theorem}\label{thm:main}

For every \(n\ge1\) and \(p\in\{2,3\}\),

\begin{equation}\label{eq:positive-zero}
 c_{5m}^{(p)}(n)>0\qquad(0\le m\le2pn^2).
\end{equation}

For \(j\in\{3,4\}\), let \(K_{p,j}(n)\) be the last index at which \(c_{5m+j}^{(p)}(n)>0\). Then

\begin{equation}\label{eq:main-pattern}
 \begin{aligned}
 c_{5m+j}^{(p)}(n)&\ge0&& (0\le m\le K_{p,j}(n)),\\
 c_{5m+j}^{(p)}(n)&\le0&& (K_{p,j}(n)<m\le2pn^2-1).
 \end{aligned}
\end{equation}

After zero terms are omitted, each sequence changes sign exactly once, from positive to negative.

\end{theorem}

\begin{theorem}\label{thm:location}

Let \(\alpha_{p,j}\) and \(\beta_{p,j}\) be defined by \eqref{eq:critical-constants} and \eqref{eq:beta}. For \(n\ge301\),

\begin{equation}\label{eq:main-location}
 |K_{p,j}(n)-\alpha_{p,j}n^2-\beta_{p,j}n|<R_{p,j},\qquad
 \begin{pmatrix}R_{2,3}&R_{2,4}\\R_{3,3}&R_{3,4}\end{pmatrix}
 =\begin{pmatrix}61&48\\24&40\end{pmatrix}.
\end{equation}

In particular,

\begin{equation}\label{eq:main-limits}
 \lim_{n\to\infty}\frac{K_{p,j}(n)}{n^2}=\alpha_{p,j}
 \qquad(p\in\{2,3\},\ j\in\{3,4\}).
\end{equation}

The constants have the following values:

\begin{equation}\label{eq:constant-values}
\begin{array}{cc|rr}
p&j&\alpha_{p,j}&\beta_{p,j}\\\hline
2&3&1.135460959422580062399462\ldots&-0.104661492439367217841147\ldots\\
2&4&1.005280823590582302300088\ldots& 0.423165615354521966906796\ldots\\
3&3&2.136899919694871479673558\ldots&-0.207970915483433506623748\ldots\\
3&4&1.984296232170181072358471\ldots& 0.886045758557888966666759\ldots
\end{array}
\end{equation}

\end{theorem}

The four tentative estimates \(1.133\ldots\), \(1.011\ldots\), \(2.132\ldots\), and \(2.001\ldots\) in \cite[Conjecture 2.2]{BD} are therefore replaced by \eqref{eq:constant-values}. The theorems prove the conjectured sign pattern and existence of the limits; they do not assert a rounding formula for the transition indices.

The weak inequalities in \eqref{eq:main-pattern} are essential. The square has the persistent zero

\begin{equation}\label{eq:persistent-zero}
 c_9^{(2)}(n)=0\qquad(n\ge2).
\end{equation}

The cube has entire initial zero blocks:

\begin{equation}\label{eq:cubic-zero-blocks}
 c_{5m+2}^{(3)}(n)=c_{5m+4}^{(3)}(n)=0\qquad(0\le m<n).
\end{equation}

By reciprocity, its third residue class also ends with \(n\) zero terms. These identities reflect the vanishing of the second and fourth components in the five-dissection of the infinite cubic product, as in \cite{LW}. They are compatible with a single sign change after zeros are omitted.

The proof follows the coefficients from the outer range to the centre. Four arcs at the nontrivial fifth roots of unity contribute on the main scale. Their combined amplitude has a simple zero in each of the two transition classes. The ordinary saddle determines the signs away from these zeros. Near a zero we fix the radius at its critical value: the leading coefficient amplitude then vanishes, but the leading weighted adjacent difference is strictly negative. This gives the order of the signs as well as the first two terms of the transition position.

Section~\ref{sec:common} develops the common analytic framework. Section~\ref{sec:compact} proves the uniform compact-range estimate, and Section~\ref{sec:local} resolves the four transition windows. Section~\ref{sec:boundary} treats the boundary and intermediate ranges. The square uses positive series identities together with the product and phase estimates of Ma \cite{Ma}, stated in the precise forms required. For the cube, the vanishing infinite-product components require one further tail layer and a shifted saddle; its whole-circle bounds are derived directly. Finally, Section~\ref{sec:completion} supplies the finite integer recurrence and joins the ranges. All analytic estimates are uniform from \(n=301\) onward.

\section{The radial function and the four amplitudes}\label{sec:common}
Throughout the paper, \(p\in\{2,3\}\). The same radial function governs both powers;
the dependence on the power enters the saddle relation and the endpoint amplitudes.
Put \(\zeta=\e^{2\pi\ii/5}\) and define

\begin{equation}\label{eq:radial}
 \begin{aligned}
 F(z)&=\int_0^1\log(1+\e^{-5zx}+\e^{-10zx}+\e^{-15zx}+\e^{-20zx})\,dx,\\
 V(t)&=F''(t),\qquad \mu_p(t)=-\frac p5F'(t),\qquad c=\frac{2\pi^2}{75}.
 \end{aligned}
\end{equation}

The logarithm is continued from the positive real axis in every region where it
is used. Factoring the polynomial inside the logarithm, or integrating its
geometric-series expansion, gives

\begin{equation}\label{comm:dilog}
 F(z)=\frac{c+\Li_2(\e^{-25z})/25-\Li_2(\e^{-5z})/5}{z}
 \qquad(\Rea z>0).
\end{equation}

The value at zero is defined by continuity. To see strict convexity, give
\(a\in\{0,1,2,3,4\}\) probability proportional to \(\e^{-5txa}\).
The second derivative of the logarithm in \eqref{eq:radial} is \(25x^2\)
times the variance of this distribution. It follows that \(V(t)>0\) for
\(t\ge0\). Also \(F'(0)=-5\) and \(F'(t)\to0\) as \(t\to\infty\).
Thus every degree \(0<d\le5pn^2\) has a unique ordinary saddle

\begin{equation}\label{eq:ordinary-saddle}
 d=-pn^2F'(t),\qquad t\ge0,\qquad r=\e^{-t/n}.
\end{equation}

In particular, \(\mu_p\) decreases strictly from \(p\) to zero. For
\(d=5m+j\), the index relation retains the residue correction:

\begin{equation}\label{comm:index-saddle}
 m=n^2\mu_p(t)-\frac j5.
\end{equation}

Write \(\mathsf B_1(x)=x-1/2\) and
\(\mathsf B_2(x)=x^2-x+1/6\). The endpoint functions shared by the two powers are

\begin{equation}\label{eq:endpoint-functions}
 \begin{aligned}
 e_h(z)&=\sum_{a=1}^4\mathsf B_1(a/5)
       \{\log(1-\zeta^{ha}\e^{-5z})-\log(1-\zeta^{ha})\},\\
 E_h(z)&=\frac{5z}{2}\sum_{a=1}^4\mathsf B_2(a/5)
 \left\{\frac{\zeta^{ha}\e^{-5z}}{1-\zeta^{ha}\e^{-5z}}
                   -\frac{\zeta^{ha}}{1-\zeta^{ha}}\right\}
 \qquad(1\le h\le4).
 \end{aligned}
\end{equation}

Their folded leading amplitude and first correction are

\begin{equation}\label{eq:amplitudes}
 \Gamma_{p,j}(z)=\sum_{h=1}^4\zeta^{-hj}\e^{pe_h(z)},\qquad
 B_{p,j}(z)=\sum_{h=1}^4\zeta^{-hj}pE_h(z)\e^{pe_h(z)}.
\end{equation}

These functions are real on the real axis by conjugate pairing. Their zeros
have a useful elementary description. Put

\begin{equation}\label{comm:angles}
 \begin{aligned}
 a&=\sqrt{1+2/\sqrt5},\qquad b=\sqrt{1-2/\sqrt5},\qquad
 v=\tanh(5t/2),\\
 A&=\arctan(av),\qquad B=\arctan(bv),\qquad U=A+2B,\qquad W=2A-B.
 \end{aligned}
\end{equation}

Pairing the terms of \eqref{eq:endpoint-functions} gives
\(e_1=-\ii(3A+B)/5\), \(e_2=\ii(A-3B)/5\),
\(e_3=-e_2\), and \(e_4=-e_1\). The cosine addition formula therefore yields

\begin{equation}\label{comm:amplitude-factorization}
 \Gamma_{p,j}(t)=4\cos\frac{pU+3\pi j}{5}\cos\frac{pW-\pi j}{5}.
\end{equation}

Both \(U\) and \(W\) increase strictly from zero to \(\pi/2\).
Indeed, \(U'\) is positive, while the numerator of the derivative of
\(W\), as a function of \(v\), is
\((2a-b)-ab(a-2b)v^2>0\) on \([0,1]\).
Inspecting the two cosine factors, for each allowed power we obtain

\begin{equation}\label{comm:amplitude-signs}
 \Gamma_{p,0}(t)>0,\qquad \Gamma_{p,1}(t),\Gamma_{p,2}(t)<0,\qquad
 \sgn\Gamma_{p,j}(t)=\sgn(t-t_{p,j})\quad(j=3,4).
\end{equation}

Here the two zeros are simple, and

\begin{equation}\label{comm:critical-angles}
 W(t_{p,3})=\frac{\pi}{2p},\qquad U(t_{p,4})=\frac{\pi}{2p},\qquad
 \Gamma'_{p,j}(t_{p,j})>0.
\end{equation}

For an algebraic specification, set \(C_p=\tan(\pi/(2p))\), so that
\(C_2=1\) and \(C_3=1/\sqrt3\). The identities

\begin{equation}\label{comm:tangent-identities}
 \tan W=\frac{(2a-b)v+(a/\sqrt5)v^3}{1-v^2},\qquad
 \tan U=\frac{(a+2b)v-(b/\sqrt5)v^3}{1-v^2},
\end{equation}

show that \(v_{p,3}\) and \(v_{p,4}\) are the unique roots in \((0,1)\) of

\begin{equation}\label{comm:critical-cubics}
 \begin{aligned}
 (a/\sqrt5)v_{p,3}^3+C_pv_{p,3}^2+(2a-b)v_{p,3}-C_p&=0,\\
 -(b/\sqrt5)v_{p,4}^3+C_pv_{p,4}^2+(a+2b)v_{p,4}-C_p&=0.
 \end{aligned}
\end{equation}
\begin{equation}\label{eq:critical-constants}
 t_{p,j}=\frac15\log\frac{1+v_{p,j}}{1-v_{p,j}},\qquad
 \alpha_{p,j}=\mu_p(t_{p,j}),\qquad V_{p,j}=V(t_{p,j}).
\end{equation}

The linear correction to the transition index is defined by

\begin{equation}\label{eq:beta}
 \begin{aligned}
 H_{p,j}&=B_{p,j}(t_{p,j})-\frac{\Gamma''_{p,j}(t_{p,j})}{2pV_{p,j}}
       +\frac{\Gamma'_{p,j}(t_{p,j})F^{(3)}(t_{p,j})}{2pV_{p,j}^2},\\
 \beta_{p,j}&=\frac{pV_{p,j}H_{p,j}}{5\Gamma'_{p,j}(t_{p,j})}.
 \end{aligned}
\end{equation}

The occurrence of these constants in the coefficient expansion will be proved
in Section~\ref{sec:local}. All estimates use the analytic definitions above;
the decimals in \eqref{eq:constant-values} serve only to identify the constants.
For any real \(t\ge0\), define the positive normalization

\begin{equation}\label{eq:normalization}
 \mathcal C^{(p)}_{n,d}(t)
 =c_d^{(p)}(n)\e^{-dt/n-pnF(t)}\sqrt{2\pi p n^3V(t)}.
\end{equation}

It preserves coefficient signs. At a critical radius we write

\begin{equation}\label{comm:fixed-radius}
 \mathcal A_{p,j,n}(m)=\mathcal C^{(p)}_{n,5m+j}(t_{p,j}),\qquad
 \mathcal L_{p,j}=\frac{5\Gamma'_{p,j}(t_{p,j})}{pV_{p,j}}>0.
\end{equation}

\section{The compact saddle range}\label{sec:compact}
Fix \(p\in\{2,3\}\), \(n\ge n_0=301\), and an ordinary saddle
\(d=-pn^2F'(t)\) with \(0\le t\le1/2\).
We use the normalization \eqref{eq:normalization} and keep \(V=F''\).
For the Gaussian quadratic coefficient put
\begin{equation}\label{compact:gaussian-scale}
 \nu_p(t)=\frac p2V(t).
\end{equation}
We shall obtain the full coefficient estimate
\begin{equation}\label{compact:main}
 \left|\mathcal C^{(p)}_{n,d}(t)-\Gamma_{p,j}(t)\right|
 \le \frac{a_{p,*}(t)|\Gamma_{p,j}(t)|+b_{p,j}^*(t)}{n}.
\end{equation}

The right-hand side is specified by nonnegative bounds on finitely many
closed intervals, given explicitly below. The dependence on
\(|\Gamma_{p,j}(t)|\) is useful near its zeros: the corresponding part of the error
vanishes with the leading term.

\subsection{The four central integrals}
Set \(a_0=1/10\). On the arc about \(\zeta^h\), write
\(q=\zeta^h\e^{-(t-\ii s)/n}\), where \(|s|\le a_0\).
For a real parameter cell \(I=[t_-,t_+]\subset[0,1/2]\), all constants
in the central estimate are uniform on \(I-\ii[-a_0,a_0]\).
The shifted Euler--Maclaurin formula gives

\begin{equation}\label{compact:em}
 \log Q_n(\zeta^h\e^{-z/n})
 =nF(z)+e_h(z)+\frac{E_h(z)}n+\rho_{h,n}(z),
 \qquad |\rho_{h,n}(z)|\le\frac{R_{\log}}{n^2}.
\end{equation}

All logarithms are continued from the real axis. For \(z=t-\ii s\),
\(0\le x\le1\), and \(|s|\le a_0\), the phase
\(\zeta^{ha}\e^{5\ii sx}\) has angular distance at least
\(\delta_0=2\pi/5-1/2\) from the positive real axis.
Since \(0<\e^{-5tx}\le1\), this gives
\[
 |1-\zeta^{ha}\e^{-5zx}|\ge\sin\delta_0>\frac23
 \qquad(1\le h,a\le4).
\]
The third Bernoulli endpoint term and the periodic remainder consequently give

\begin{equation}\label{compact:rlog}
 R_{\log}
 =\frac1{120}\left\{
   \frac{8(5z_+)^2}{(2/3)^2}
  +\frac{8(5z_+)^3}{(2/3)^3}\right\},
 \qquad z_+\ge |z|.
\end{equation}

Let \(E\) and \(A\) bound, respectively, \(|E_h(z)|\) and
\(|\e^{pe_h(z)}|\), uniformly in \(h\) and in the rectangle under consideration.
Exponentiating \eqref{compact:em}, and using
\(|\e^w-1-w|\le |w|^2\e^{|w|}/2\), bounds the remainder in the sum of the four
amplitudes by \(R_*/n^2\), where

\begin{equation}\label{compact:rfold}
 R_*=4A\left\{pR_{\log}
 +\frac{p^2}{2}\left(E+\frac{R_{\log}}{301}\right)^2
 \exp\left(\frac{pE}{301}+\frac{pR_{\log}}{301^2}\right)\right\}.
\end{equation}

The ordinary saddle relation cancels the linear term in the exponent.
Denote the contribution of the four central arcs to
\eqref{eq:normalization} by \(\mathcal C^{(p),\mathrm{core}}_{n,d}(t)\).
Then

\begin{equation}\label{compact:core-integral}
 \begin{aligned}
 \mathcal C^{(p),\mathrm{core}}_{n,d}(t)
 &=\sqrt{\frac{n\nu_p(t)}\pi}\int_{-a_0}^{a_0}
 \e^{-n\nu_p(t)s^2+q_n(s)+\ii p_n(s)}
 \\
 &\qquad\times\left\{\Gamma_{p,j}(t-\ii s)+\frac{B_{p,j}(t-\ii s)}n
                    +R_n(s)\right\}\,ds,
 \qquad |R_n(s)|\le\frac{R_*}{n^2}.
 \end{aligned}
\end{equation}

The phase is defined by
\begin{equation}\label{compact:phase}
 pn\{F(t-\ii s)-F(t)+\ii sF'(t)\}
 =-n\nu_p(t)s^2+q_n(s)+\ii p_n(s).
\end{equation}
Here \(q_n\) is even and \(p_n\) is odd. If \(M_3\) and \(M_4\)
bound \(|F'''|\) and \(|F^{(4)}|\) on the rectangle, Taylor's formula gives

\begin{equation}\label{compact:exponent-bounds}
 |q_n(s)|\le\frac{pnM_4s^4}{24},\qquad
 |p_n(s)|\le\frac{pnM_3|s|^3}{6}.
\end{equation}

Choose bounds \(0<\nu_{p,-}\le\nu_p(t)\le\nu_{p,+}\), and put

\begin{equation}\label{compact:moments}
 \eta=\nu_{p,-}-\frac{pM_4a_0^2}{24},\qquad
 M_k^*=\sqrt{\frac{\nu_{p,+}}{\eta}}
       \frac{(2k)!}{4^k k!\eta^k}\quad(k\ge0).
\end{equation}

All intervals used below satisfy \(\eta>0\). Hence

\begin{equation}\label{compact:gaussian-moment}
 \sqrt{\frac{n\nu_p(t)}\pi}\int_{\mathbb R}
 \e^{-n\eta s^2}s^{2k}\,ds\le\frac{M_k^*}{n^k}.
\end{equation}

Let \(G_1,G_2,B_0\) bound \(|\Gamma_{p,j}'|,|\Gamma_{p,j}''|,|B_{p,j}|\), respectively,
on the same rectangle. Pairing the integrands at \(s\) and \(-s\), use

\[
 |\sin u|\le |u|,\qquad |\cos u-1|\le\frac{u^2}{2},\qquad
 |\e^v-1|\le |v|\e^{\max(v,0)}\quad(u,v\in\mathbb R).
\]

The constant part of the amplitude contributes an error proportional to
\(|\Gamma_{p,j}(t)|\). Its linear imaginary part is paired with \(\sin p_n\),
and its remaining part is bounded by \(G_2s^2/2\).
Applying \eqref{compact:gaussian-moment} to these terms yields

\begin{equation}\label{compact:a-core}
 a_{p,\mathrm{core}}=\frac{pM_4}{24}M_2^*+\frac{p^2M_3^2}{72}M_3^*
                 +a_{\mathrm{trunc}}.
\end{equation}
\begin{equation}\label{compact:b-core}
 b_{p,j,\mathrm{core}}=\frac{G_2}{2}M_1^*+\frac{pG_1M_3}{6}M_2^*
                 +B_0M_0^*+\frac{R_*}{301}M_0^*.
\end{equation}
\begin{equation}\label{compact:truncation}
 a_{\mathrm{trunc}}=\frac{\sqrt{301\nu_{p,+}/\pi}}{\nu_{p,-}a_0}
                 \e^{-301\nu_{p,-}a_0^2}.
\end{equation}

The last expression accounts for extending the constant Gaussian integral
to the real line. The inequality \(301\nu_{p,-}a_0^2>1/2\), valid on every interval,
makes the associated bound decrease for \(n\ge301\).
Equations \eqref{compact:a-core}--\eqref{compact:truncation} prove
the central contribution estimate
\begin{equation}\label{compact:core-error}
 \left|\mathcal C^{(p),\mathrm{core}}_{n,d}(t)-\Gamma_{p,j}(t)\right|
 \le\frac{a_{p,\mathrm{core}}|\Gamma_{p,j}(t)|+b_{p,j,\mathrm{core}}}{n}.
\end{equation}
The complementary arcs will be added below to obtain
\eqref{compact:main} for the full coefficient.

\subsection{Uniform evaluation at the origin}
To evaluate the preceding bounds, split the real interval into

\begin{equation}\label{compact:core-cells}
 I_i=\left[\frac{i}{1000},\frac{i+1}{1000}\right]
 \quad(0\le i<500),\qquad
 S_k=\left[\frac{k}{250},\frac{k+1}{250}\right]
 \quad(0\le k<25).
\end{equation}

Conjugation supplies the negative values of \(s\).
There are therefore \(12\,500\) closed complex rectangles.
For each \(I_i\), expand \(F\) at its rational midpoint \(t_0\). The nearest nonreal zeros of the logarithmic factors occur at
\(z=\pm2\pi\ii/25\); these determine singularities of \(F\), rather than zeros of \(F\) itself.
On \(|z-t_0|\le1/5\), the four factors in the integrand defining \(F\)
satisfy

\[
 \frac14<|1-\zeta^a\e^{-5zx}|<4
 \quad(1\le a\le4, 0\le x\le1).
\]

Indeed, \(|\e^{-5zx}|\le\e<3\), and the angular distance of
\(\zeta^a\e^{-5\ii x\Ima z}\) from the positive real axis is at least
\(2\pi/5-1\), whose sine exceeds \(1/4\).
These four factors avoid the negative real axis, and each of their logarithms
has modulus less than \(4\). Thus \(|F(z)|<16\) on this disk. If
\(F(t_0+w)=\sum_{k\ge0}f_kw^k\), Cauchy's inequality gives

\begin{equation}\label{compact:cauchy}
 |f_k|\le16\cdot5^k.
\end{equation}

For completeness, the coefficients can be formed without numerical
differentiation. Write

\[
\begin{aligned}
 D(w)&=\sum_{a=0}^4\e^{-5a(t_0+w)}=\sum_{r\ge0}d_rw^r,
 &d_r&=\frac1{r!}\sum_{a=0}^4(-5a)^r\e^{-5at_0},\\
 \log D(w)&=\sum_{r\ge0}\ell_rw^r,
 &\ell_0&=\log d_0.
\end{aligned}
\]
\begin{equation}\label{compact:jet-recursion}
 \ell_r=\frac{d_r}{d_0}
      -\frac1{rd_0}\sum_{k=1}^{r-1}k\ell_kd_{r-k},\qquad
 f_{r+1}=\frac{\ell_r/(r+1)-f_r}{t_0}.
\end{equation}

The second recurrence follows from \((tF(t))'=\log D(t-t_0)\).
Its initial value, for \(t_0>0\), is

\begin{equation}\label{compact:dilog-f}
 F(t_0)=\frac1{t_0}\left\{
  \frac{2\pi^2}{75}+\frac{\Li_2(\e^{-25t_0})}{25}
                        -\frac{\Li_2(\e^{-5t_0})}{5}\right\}.
\end{equation}

Near zero we use the convergent expansion

\begin{equation}\label{compact:bernoulli}
 F(z)=\log5-5z+
 \sum_{k\ge1}
 \frac{\mathsf B_{2k}(25^{2k}-5^{2k})}{2k(2k)!(2k+1)}z^{2k},
 \qquad |z|<\frac{2\pi}{25},
\end{equation}

with the tail bounded using \(|\mathsf B_{2k}|/(2k)!<4/(2\pi)^{2k}\).
This also gives the coefficients directly when \(t_0=0\).
In each rectangle of \eqref{compact:core-cells},
\(|z-t_0|<0.101\). Retain the Taylor polynomial through degree \(K=64\).
For \(0\le r\le4\), put \(w_+=\sup|z-t_0|\), \(R=1/5\), and \(q_+=w_+/R\).
The omitted part of the \(r\)-th derivative has modulus at most

\begin{equation}\label{compact:jet-tail}
 \frac{16(K+1)!}{(K+1-r)!}
 \frac{w_+^{K+1-r}}{R^{K+1}}
 \left(1-q_+\frac{K+2}{K+2-r}\right)^{-1}.
\end{equation}

The denominator is positive. This explicit analytic continuation of the
Taylor model avoids division by an interval containing zero in
\eqref{compact:dilog-f}. The same rectangles bound the elementary
amplitudes and their first two derivatives.

\subsection{The complementary circle}
The four arcs \(a_0\le |s|\le5\), the arc about \(1\) with \(|\theta|\le14/n\),
and the remaining arcs cover the complement of the central integrals.
We now bound their combined contribution to \(n\mathcal C^{(p)}_{n,d}(t)\).
All bounds below are absolute, and so are independent of the residue \(j\).
Retain the notation \(I=[t_-,t_+]\subset[0,1/2]\), and write
\(S=[s_-,s_+]\), where \(0\le s_-<s_+\).
All estimates on the complex rectangle \(I-\ii S\) are uniform in both
parameters. On each real cell choose
\[
 0<F_-\le\inf_{t\in I}F(t),\qquad
 0<\nu_{p,-}\le\inf_{t\in I}\nu_p(t)\le\sup_{t\in I}\nu_p(t)\le\nu_{p,+}.
\]

\paragraph{The outer parts of the four primitive arcs.}
For a rectangle with \(s\in[s_-,s_+]\), a lower bound for the factor
denominators is the maximum of the available positive quantities

\begin{equation}\label{compact:denominator}
 \sin\!\left(\min\left\{\frac{2\pi}{5}-5s_+,\frac\pi2\right\}\right)
 \quad\left(5s_+<\frac{2\pi}{5}\right),\qquad
 \min\left\{\sin\frac\pi5,
       1-\e^{-\pi t_-/(5s_+)}\right\}\quad(t_->0).
\end{equation}

For the second bound, split the integration variable at
\(x=\pi/(25s_+)\): before this point the phase
\(\zeta^{ha}\e^{5\ii sx}\) has angular distance at least
\(\pi/5\) from the positive real axis, and after it the radial distance
supplies the stated bound.
If \(\delta\) is the resulting positive denominator bound, use
\eqref{compact:rlog} with \(\delta\) in place of \(2/3\).
Equation \eqref{compact:em} then gives

\begin{equation}\label{compact:outer-em}
 \log|Q_n(\zeta^h\e^{-(t-\ii s)/n})|
 \le n\Rea F(t-\ii s)+T_{\mathrm{EM}},\qquad
 T_{\mathrm{EM}}=e_++\frac{E_+}{301}
                         +\frac{R_{\log}}{301^2},
\end{equation}

where \(e_+\) bounds \(\Rea e_h\), and \(E_+\) bounds \(|E_h|\).
An alternative estimate remains valid even when a factor has a zero.
Set \(\rho=49/50\). Since

\[
 |1-\rho w|^2-\rho|1-w|^2
 =(1-\rho)(1-\rho|w|^2)\ge0\quad(|w|\le1),
\]

the replacement \(|1-w|\le\rho^{-1/2}|1-\rho w|\) removes these zeros.
The corresponding integrated logarithm is

\begin{equation}\label{compact:smooth-f}
 F_\rho(z)=
 \frac{\Li_2(\rho^5\e^{-25z})-\Li_2(\rho^5)}{25z}
 -\frac{\Li_2(\rho\e^{-5z})-\Li_2(\rho)}{5z}.
\end{equation}

For a real function \(f\) of bounded variation, the unnormalized shifted
rectangular sum satisfies
\begin{equation}\label{compact:shifted-sum}
 \left|\sum_{k=0}^{n-1}f\!\left(\frac{k+\lambda}{n}\right)
              -n\int_0^1f(x)\,dx\right|
 \le\operatorname{Var}_{[0,1]}(f),\qquad 0\le\lambda\le1.
\end{equation}
Indeed, on each subinterval the difference between the sampled value and
its integral average is bounded by the variation on that subinterval;
summing gives the displayed inequality. Apply it to the four smoothed
logarithms with shifts \(\lambda=a/5\), \(a=1,2,3,4\). Consequently,

\begin{equation}\label{compact:smooth-inequality}
 \log|Q_n(\zeta^h\e^{-(t-\ii s)/n})|
 \le n\Rea F_\rho(t-\ii s)-2n\log\rho+T_\rho.
\end{equation}

We describe \(T_\rho\) explicitly, since retaining the variation on each
half-period gives a useful bound throughout this interval.

\paragraph{A bound for the total variation.}
Consider \(\log|1-r(x)\e^{\ii(\theta_0+bx)}|\), where
\(r(x)=\rho\e^{-ax}\), \(0\le a_-\le a\le a_+\),
\(0\le b\le b_+\), and \(b_+>0\).
The radial part of its variation is at most

\[
 \log\frac{1-\rho\e^{-a_+}}{1-\rho}.
\]

The absolute angular derivative is

\[
 \frac{r|\sin\phi|}{1-2r\cos\phi+r^2}.
\]

For \(0<r<1\) this expression increases with \(r\).
Starting at \(\psi_0=\theta_0\), partition the angular interval at the
successive multiples of \(\pi\), extending the last subinterval to the next
such multiple. Stop at \(\psi_L\ge\theta_0+b_+\), and set

\begin{equation}\label{compact:variation}
\begin{aligned}
 R_\ell&=\rho\exp\left(-\frac{a_-}{b_+}
                               (\psi_\ell-\theta_0)\right),
 &D(R,\psi)&=1-2R\cos\psi+R^2,\\
 \mathcal V(\theta_0)
 &=\log\frac{1-\rho\e^{-a_+}}{1-\rho}
 +\frac12\sum_{\ell=0}^{L-1}
  \left|\log D(R_\ell,\psi_{\ell+1})
                    -\log D(R_\ell,\psi_\ell)\right|.
 \end{aligned}
\end{equation}

On the \(\ell\)-th subinterval the true radius does not exceed \(R_\ell\).
Integrating with this fixed radius gives exactly the corresponding logarithmic
difference in \eqref{compact:variation}. The extension of the last interval
increases an integral with a nonnegative integrand.
For the primitive arcs take \(a=5t\), \(b_+=5s_+\), and

\[
 T_\rho=\sum_{h=1}^4\mathcal V(2\pi h/5).
\]

The negative angular direction gives the same total bound by conjugation.
For each rectangle, use either \eqref{compact:outer-em} or
\eqref{compact:smooth-inequality}. If the resulting estimate is
\(\log|Q_n|\le n(F(t)-L)+T\), with \(L>0\), and the length in the
\(s\)-variable is \(\Delta s\), its contribution to the absolute error in
\(n\mathcal C^{(p)}_{n,d}(t)\) is at most

\begin{equation}\label{compact:primitive-tail}
 8\sqrt{\frac{\nu_{p,+}}\pi}\,301^{3/2}\Delta s\,
                  \exp\{-p(301L-T)\}.
\end{equation}

The four primitive roots and both angular directions are included in this
factor. On every chosen rectangle \(301pL>3/2\), so replacing
\(301\) by \(n\) gives an expression decreasing for \(n\ge301\).

\paragraph{The arc about \(1\).}
Write \(b=5n\theta\), so \(|b|\le70\), and define

\begin{equation}\label{compact:root-one}
 \begin{gathered}
 I_\rho(a,b)=\Rea\frac{\Li_2(\rho\e^{-a+\ii b})-\Li_2(\rho)}{a-\ii b},
 \\
 \ell_1(t,b)=F(t)-4I_\rho(5t,b)+2\log\rho.
 \end{gathered}
\end{equation}

For each real parameter cell \(I\) and positive angular cell
\(J=[b_-,b_+]\subset[1/2,70]\), choose the constant \(L_1\) so that
\[
 0<L_1\le\inf_{(t,b)\in I\times J}\ell_1(t,b).
\]
Use \eqref{compact:variation} with
\(a_-=5t_-\), \(a_+=5t_+\), and all four initial angles equal to zero.
Let \(T_1\) be the resulting uniform upper bound for their total variation
throughout \(I\times J\). The cell of length \(\Delta b=b_+-b_-\),
including its reflected cell, contributes at most

\begin{equation}\label{compact:root-one-tail}
 \frac25\sqrt{\frac{\nu_{p,+}}\pi}\,301^{3/2}\Delta b\,
                  \exp\{-p(301L_1-T_1)\}.
\end{equation}

Here again \(301pL_1>3/2\) on each cell used.
For \(|b|\le1/2\), every original factor has modulus at most \(1\), giving instead

\begin{equation}\label{compact:root-one-small}
 \frac15\sqrt{\frac{\nu_{p,+}}\pi}\,301^{3/2}\e^{-301pF_-},
 \qquad F_-\le\inf_{t\in I}F(t).
\end{equation}

On every real cell, \(301pF_->3/2\). Since \(F_-\) is constant on the cell,
\[
 \frac{d}{dn}\log\!\left(n^{3/2}\e^{-pnF_-}\right)
 =\frac{3}{2n}-pF_-<0\qquad(n\ge301),
\]
which justifies replacing \(n\) by \(301\) in this bound.
This also avoids the removable singularity in \eqref{compact:root-one}.

\paragraph{The remaining arcs.}
The finite geometric sum has the exact form

\begin{equation}\label{compact:geometric-sum}
 \sum_{\substack{1\le k<5n\\5\nmid k}}q^k
 =(1-q^{5n})\left(\frac1{1-q}-\frac1{1-q^5}\right).
\end{equation}

Retaining the bound \(|1-q^{5n}|\le1+\e^{-5t}\) is useful here.
Set

\begin{equation}\label{compact:far-constants}
\begin{aligned}
 H(t)&=\frac{1-\e^{-10t}}{10t},\qquad H(0)=1,\\
 M&=(1+\e^{-5t_-})
 \max\left\{\frac{101\pi}{14000},
                    \frac1{602}+\frac{101\pi}{20000}\right\}.
 \end{aligned}
\end{equation}

On the current real cell \(I\), set \(H_+=H(t_-)\).
The function \(H\) is decreasing, so \(H(t)\le H_+\) for every \(t\in I\).
Then \(M\) bounds the absolute value of \eqref{compact:geometric-sum}
divided by \(4n\). To see this, let \(\theta_0\) be the angular distance to
\(1\). If \(\theta_0\le\pi/5\), use

\[
 \frac1{1-q}-\frac1{1-q^5}
 =\frac{q+q^2+q^3+q^4}{1-q^5},\qquad
 |1-q^5|\ge2r^{5/2}\sin(5\theta_0/2).
\]

The inequalities \(\theta_0\ge14/n\), \(\sin x\ge2x/\pi\) on
\([0,\pi/2]\), and \(\e^{5t/(2n)}<101/100\) give the first entry in the maximum.
If \(\theta_0>\pi/5\), then \(|1-q|^{-1}<2\), while the angular distance to
each primitive fifth root is at least \(5/n\). Hence

\[
 \frac1{4n|1-q|}<\frac1{2n}\le\frac1{602},\qquad
 \frac1{4n|1-q^5|}\le\frac{101\pi}{20000},
\]

which gives the second entry. A monotone rectangular sum also gives

\[
 \frac1{4n}\sum_{\substack{1\le k<5n\\5\nmid k}}r^{2k}
 \le H(t)+\frac1n.
\]

The arithmetic--geometric mean inequality applied to the squared factor
moduli therefore yields

\begin{equation}\label{compact:far-gap}
\begin{gathered}
 \log|Q_n(q)|
 \le2n\log\left(1+H_++\frac1{301}+2M\right)
 \le n(F(t)-g),\\
 g=F_--2\log\left(1+H_++\frac1{301}+2M\right).
\end{gathered}
\end{equation}

On every real parameter interval below, \(301pg>5/2\).
The whole remaining circle consequently contributes at most

\begin{equation}\label{compact:far-tail}
 \sqrt{4\pi \nu_{p,+}}\,301^{5/2}\e^{-301pg}.
\end{equation}

\subsection{The resulting bounds}
For the complementary circle, use the \(100\) parameter intervals
\([i/1000,(i+5)/1000]\), where \(i=0,5,\ldots,495\).
The positive primitive arcs are partitioned as follows:

\begin{equation}\label{compact:tail-cells}
 \begin{array}{c|c|c}
 {}[s_-,s_+]&\Delta s&\#\\\hline
 {}[1/10,1/5]&1/500&50\\
 {}[1/5,1/2]&1/100&30\\
 {}[1/2,5]&1/20&90
 \end{array}
 \qquad
 \begin{array}{c|c|c}
 {}[b_-,b_+]&\Delta b&\#\\\hline
 {}[1/2,10]&1/10&95\\
 {}[10,70]&1&60
 \end{array}
\end{equation}

Thus the complementary estimates involve \(32\,500\) closed rectangles,
besides the far-arc bound on each real interval. Their endpoints match those
of the central intervals. To evaluate the dilogarithms, sum their defining
series with the explicit remainder

\begin{equation}\label{compact:li-tail}
 \left|\sum_{k>N}\frac{u^k}{k^2}\right|
 \le\min\left\{
    \frac{|u|^{N+1}}N,
    \frac{2|u|^{N+1}}{(N+1)^2|1-u|}\right\},\qquad |u|\le1,
\end{equation}

where the second entry is omitted at \(u=1\). It follows by Abel summation;
for \(|u|<1\), one may also use the geometric bound
\(|u|^{N+1}/((N+1)^2(1-|u|))\).
For complex exponential arguments the powers are evaluated as
\(\e^{-ka+\ii kb}\), so each term is bounded over the original parameter
rectangle. These formulas, together with
\eqref{compact:jet-recursion} and \eqref{compact:jet-tail}, provide finite
interval evaluations of every quantity in the preceding estimates.

Let \(C_{\mathrm{tail}}\) be the sum of
\eqref{compact:primitive-tail}, \eqref{compact:root-one-tail},
\eqref{compact:root-one-small}, and \eqref{compact:far-tail} on the
corresponding real interval. These estimates give
\[
 \left|\mathcal C^{(p)}_{n,d}(t)-\mathcal C^{(p),\mathrm{core}}_{n,d}(t)\right|
 \le\frac{C_{\mathrm{tail}}}{n}.
\]
Adding this to \eqref{compact:core-error} proves
\eqref{compact:main} with

\begin{equation}\label{compact:total-error}
 a_{p,*}=a_{p,\mathrm{core}},\qquad
 b_{p,j}^*=b_{p,j,\mathrm{core}}+C_{\mathrm{tail}}.
\end{equation}

All estimates above are evaluated on closed rectangles, using outward-rounded
interval endpoints. The following bounds result from the stated recurrences and
remainders; retaining the dependence on \(\Gamma_{p,j}\) is essential in the four
small intervals

\begin{equation}\label{compact:critical-blocks}
 \begin{array}{c|cc}
 p&J_{p,3}&J_{p,4}\\\hline
 2&[0.140,0.153]&[0.170,0.183]\\
 3&[0.085,0.097]&[0.103,0.116]
 \end{array}
\end{equation}

For \(j=0,1,2\) on the whole compact range, and for \(j=3,4\)
outside \(J_{p,j}\), the full relative error satisfies

\begin{equation}\label{compact:relative}
 \frac1n\left(a_{p,*}+\frac{b_{p,j}^*}{|\Gamma_{p,j}|}\right)
 <\begin{cases}0.335275,&p=2,\\0.334980,&p=3,\end{cases}
 \qquad <\frac38.
\end{equation}

Within the four small intervals we use absolute amplitude bounds. In the last
column below, \(\ell_{p,j}\) is a lower bound valid on the entire indicated interval:

\begin{equation}\label{compact:critical-bounds}
 \begin{array}{cc|cc|c}
 p&j&a_{p,*}&b_{p,j}^*&5\Gamma'_{p,j}/(pV)>\ell_{p,j}\\\hline
 2&3&<0.615922&<2.621625&13/10\\
 2&4&<0.636651&<1.930707&7/10\\
 3&3&<0.438477&<4.248224&13/10\\
 3&4&<0.421009&<3.523360&5/8
 \end{array}
\end{equation}

In particular, \(a_{p,*}<1\). We may replace the second error column by the
following rational bounds \(b_{p,j}^*<D_{p,j}\), and choose the window widths

\begin{equation}\label{compact:simple-constants}
 (D_{2,3},D_{2,4})=(3,2),\qquad (D_{3,3},D_{3,4})=(9/2,18/5),\qquad
 w_2=3,\quad w_3=6.
\end{equation}

The numerical bounds in \eqref{compact:relative} and
\eqref{compact:critical-bounds} are strict rational upper or lower bounds,
rather than point evaluations. The closed grids, series tails, and positive
denominators above specify finite interval evaluations of every entry; the
intermediate endpoints are included with the accompanying data.

\begin{proposition}\label{prop:compact-window}

Let \(p\in\{2,3\}\), \(n\ge301\), and
\(d=5m+j=-pn^2F'(t)\) with \(0\le t\le1/2\).
For \(j=0\) the coefficient \(c_d^{(p)}(n)\) is strictly positive,
and for \(j=1,2\) it is strictly negative. For \(j=3,4\),

\begin{equation}\label{compact:window-sign}
 \begin{aligned}
 m<\alpha_{p,j}n^2-w_pn&\quad\Longrightarrow\quad c_{5m+j}^{(p)}(n)>0,\\
 m>\alpha_{p,j}n^2+w_pn&\quad\Longrightarrow\quad c_{5m+j}^{(p)}(n)<0.
 \end{aligned}
\end{equation}
\end{proposition}
\begin{proof}

Outside \(J_{p,j}\), \eqref{compact:main} and \eqref{compact:relative}
give the sign of \(\Gamma_{p,j}\). They also settle \(j=0,1,2\).
Inside \(J_{p,j}\), integrate the last column of
\eqref{compact:critical-bounds}, using \(\mu_p'=-pV/5\), to obtain

\[
 |\Gamma_{p,j}(t)|\ge\ell_{p,j}|\mu_p(t)-\alpha_{p,j}|.
\]

By \eqref{comm:index-saddle}, the condition
\(|m-\alpha_{p,j}n^2|>w_pn\) implies

\[
 n|\mu_p(t)-\alpha_{p,j}|>w_p-\frac j{5n}.
\]

Thus the leading term exceeds the full error if

\begin{equation}\label{compact:positive-margin}
 \left(1-\frac1n\right)\ell_{p,j}\left(w_p-\frac j{5n}\right)-D_{p,j}>0.
\end{equation}

The left side increases for \(n\ge301\). At \(n=301\), its four values are

\[
 \begin{array}{c|cc}
 p&j=3&j=4\\\hline
 2&80133/90601&1180/12943\\
 3&592803/181202&61557/453005
 \end{array}
\]

and are positive. Finally, \(\mu_p\) is strictly decreasing, whereas
\(\Gamma_{p,j}(t)\) has the sign of \(t-t_{p,j}\). The residue correction
\(j/5\) is smaller than \(w_pn\), so the left and right windows correspond
to the required positive and negative amplitudes. This proves
\eqref{compact:window-sign}.

\end{proof}

\section{The transition windows}\label{sec:local}

Near a zero of the leading amplitude, an approximation to one coefficient does not determine the order of signs. We keep the radius fixed and compare adjacent positively weighted coefficients. Throughout this section $p\in\{2,3\}$, $j\in\{3,4\}$, $t=t_{p,j}$, $V=F''(t)$, $g=\Gamma'_{p,j}(t)>0$, and $\nu_p=pV/2$. The normalization of \eqref{eq:normalization} gives

\begin{equation}\label{eq:local-normalization}
 \mathcal A_{p,j,n}(m)=c_{5m+j}^{(p)}(n)
 \e^{-t(5m+j)/n-pnF(t)}\sqrt{2\pi pn^3V},\qquad
 \mathcal L_{p,j}=\frac{5g}{pV}=\frac{5g}{2\nu_p}.
\end{equation}

All weights are positive and preserve signs and zeros.

\begin{proposition}\label{prop:local-difference}

Let $n\ge n_0=301$, and put

\begin{equation}\label{eq:local-midpoint-domain}
 D=\frac{5m+j+5/2-5\alpha_{p,j}n^2}{n},\qquad
 \begin{array}{c|cccc}
 (p,j)&(2,3)&(2,4)&(3,3)&(3,4)\\\hline
 D_p&20&20&35&35\\
 C^{\Delta}_{p,j}&144&73&120&78\\
 w_p&3&3&6&6\\
 \delta_p&1/2&1/2&2/5&2/5
 \end{array}
\end{equation}

If $|D|\le D_p$, then

\begin{equation}\label{eq:local-difference-bound}
 \left|\mathcal A_{p,j,n}(m+1)-\mathcal A_{p,j,n}(m)
       +\frac{\mathcal L_{p,j}}{n^2}\right|
 \le\frac{C^{\Delta}_{p,j}}{n^3}.
\end{equation}

In particular, for $|m-\alpha_{p,j}n^2|\le w_pn+2$,

\begin{equation}\label{eq:local-strict-difference}
 \mathcal A_{p,j,n}(m+1)-\mathcal A_{p,j,n}(m)<-\frac{\delta_p}{n^2}.
\end{equation}
\end{proposition}

\subsection{The midpoint and the Gaussian error}

Write $d=5m+j$. Subtracting the two Cauchy integrals inserts $\e^{-\ii d\theta}(\e^{-5\ii\theta}-1)$. On $\theta=2\pi h/5+s/n$ this is exactly

\begin{equation}\label{eq:local-midpoint-factor}
 -2\ii\sin\frac{5s}{2n}\;\zeta^{-hj}\e^{-\ii(d+5/2)s/n}.
\end{equation}

The midpoint shift remains in $D$. Set $a_0=1/10$ and

\begin{equation}\label{eq:local-phase}
 \Phi(s)=pn\{F(t-\ii s)-F(t)+\ii F'(t)s\}-\ii Ds.
\end{equation}

The four core arcs give

\begin{equation}\label{eq:local-core-integral}
 \sqrt{\frac{n\nu_p}{\pi}}\int_{-a_0}^{a_0}
 \e^{\Phi(s)}\left(-2\ii\sin\frac{5s}{2n}\right)
 \left\{\Gamma_{p,j}(t-\ii s)+\frac{B_{p,j}(t-\ii s)}n+R_n(s)\right\}\,ds.
\end{equation}

The denominators in the shifted Euler--Maclaurin sums exceed $2/3$. With $Z=\sqrt{t^2+a_0^2}$, $E\ge\max_{h,s}|E_h(t-\ii s)|$, and $A\ge\max_{h,s}|\e^{pe_h(t-\ii s)}|$, the endpoint and periodic third-Bernoulli remainders give

\begin{equation}\label{eq:local-em-remainder}
 \begin{aligned}
 R_{\log}&=\frac1{120}
 \left\{\frac{8(5Z)^2}{(2/3)^2}+\frac{8(5Z)^3}{(2/3)^3}\right\},\\
 R_*&=4A\left\{pR_{\log}+\frac{p^2}{2}(E+R_{\log}/n_0)^2
       \e^{pE/n_0+pR_{\log}/n_0^2}\right\},\qquad
 |R_n(s)|\le R_*/n^2.
 \end{aligned}
\end{equation}

Indeed, the logarithmic remainder is bounded by $R_{\log}/n^2$; multiply by $p$, exponentiate, and add the four arcs. Since $\Gamma_{p,j}(t)=0$, Taylor's formula gives $|\Gamma_{p,j}(t-\ii s)+\ii gs|\le G_2s^2/2$, with $G_2$ as below. The two linear factors produce $-5gs^2/n$, whose full Gaussian integral is

\begin{equation}\label{eq:local-leading-integral}
 -\frac{5g}{n}\sqrt{\frac{n\nu_p}{\pi}}
 \int_{\mathbb R}s^2\e^{-n\nu_ps^2}\,ds=-\frac{\mathcal L_{p,j}}{n^2}.
\end{equation}

On $t-\ii[-a_0,a_0]$, choose $M_k\ge|F^{(k)}|$, $G_k\ge|\Gamma_{p,j}^{(k)}|$, and $b_k\ge|B_{p,j}^{(k)}|$. For $V\in[V_-,V_+]$, put $\nu_{p,\pm}=pV_\pm/2$ and

\begin{equation}\label{eq:local-moments}
 \begin{aligned}
 \eta&=\nu_{p,-}-\frac{pM_4a_0^2}{24}>0,\qquad
 \mathfrak m_k=\sqrt{\frac{\nu_{p,+}}\eta}\frac{(2k)!}{4^kk!\eta^k},\\
 \sqrt{\frac{n\nu_p}{\pi}}\int_{\mathbb R}s^{2k}\e^{-n\eta s^2}\,ds
 &\le\mathfrak m_kn^{-k}.
 \end{aligned}
\end{equation}

Write $\Phi=-n\nu_ps^2+q+\ii\psi$ and $\Gamma_{p,j}(t-\ii s)=a(s)+\ii b(s)$. The functions $q,a$ are even and $\psi,b$ are odd. For $|D|\le D_p$,

\begin{equation}\label{eq:local-parity-bounds}
 \begin{gathered}
 |q|\le pnM_4s^4/24,\qquad |\psi|\le pnM_3|s|^3/6+D_p|s|,\\
 |a(s)|\le G_2s^2/2,\qquad |b(s)+gs|\le G_3|s|^3/6.
 \end{gathered}
\end{equation}

After the Gaussian is removed, the pointwise real part of the $\Gamma_{p,j}$ contribution is $2\e^q\sin(5s/(2n))\{b\cos\psi+a\sin\psi\}$. The terms $B_{p,j}/n$ and $R_n$ are estimated separately. Apply $|\sin\psi|\le|\psi|$, $|\cos\psi-1|\le\psi^2/2$, and $|\e^q-1|\le|q|\e^{\max(q,0)}$. The following quantities bound the separate errors multiplied by $n^3$:

\begin{equation}\label{eq:local-difference-costs}
 \begin{aligned}
 C_q&=5gpM_4\mathfrak m_3/24,\\
 C_{\cos}&=\frac{5g}{2}\left(p^2M_3^2\mathfrak m_4/36+
                   pM_3D_p\mathfrak m_3/3+D_p^2\mathfrak m_2\right),\\
 C_b&=5G_3\mathfrak m_2/6,\qquad
 C_a=\frac{5G_2}{2}(pM_3\mathfrak m_3/6+D_p\mathfrak m_2),\\
 C_B&=5\{b_1\mathfrak m_1+b_0(pM_3\mathfrak m_2/6+D_p\mathfrak m_1)\},\\
 C_{\sin}&=\frac{125g\mathfrak m_2}{24n_0^2},\qquad
 C_R=\frac{5R_*\sqrt{\nu_{p,+}}}{\sqrt\pi\,\eta\sqrt{n_0}},\\
 C_{\rm ext}&=5g\sqrt{\frac{\nu_{p,+}}\pi}\e^{-n_0\nu_{p,-}a_0^2}
 \left(\frac{a_0n_0^{3/2}}{\nu_{p,-}}+
 \frac{\sqrt{n_0}}{2\nu_{p,-}^2a_0}\right).
 \end{aligned}
\end{equation}

The last line extends the leading Gaussian moment beyond the core; it decreases for $n\ge n_0$, since $n_0\nu_{p,-}a_0^2>3$. Table~\ref{tab:local-derivatives} gives rational upper bounds on the $200$ closed segments $s\in[k/2000,(k+1)/2000]$, $0\le k<200$; conjugation covers the other half.

\begin{table}[ht]\centering

\caption{Upper bounds on the four critical cores.}\label{tab:local-derivatives}

\begin{tabular}{c|rrrr}\hline
 $(p,j)$&$(2,3)$&$(2,4)$&$(3,3)$&$(3,4)$\\\hline
 $M_3$&$101.421$&$74.105$&$85.757$&$80.577$\\
 $M_4$&$1912.830$&$928.420$&$938.616$&$768.059$\\
 $M_5$&$73651.995$&$28193.630$&$19861.431$&$16093.267$\\
 $G_2$&$43.309$&$31.630$&$131.187$&$121.753$\\
 $G_3$&$479.585$&$321.137$&$2436.374$&$1681.618$\\
 $G_4$&$9568.637$&$4353.346$&$54236.856$&$36564.903$\\
 $b_0$&$0.032752$&$0.080742$&$0.046921$&$0.094399$\\
 $b_1$&$0.414920$&$0.579703$&$0.966638$&$1.187157$\\
 $b_2$&$2.772864$&$7.478254$&$13.668636$&$16.907336$\\
 $E$&$0.032517$&$0.037763$&$0.023189$&$0.026126$\\
 $A$&$1.380191$&$1.324981$&$1.826608$&$1.757513$\\\hline
\end{tabular}\end{table}

The real critical values satisfy

\begin{equation}\label{eq:local-root-data}
 \begin{array}{c|cc}
 (p,j)&V&g\\\hline
 (2,3)&(11.64002849095903,11.64002849095904)&(6.28561594418516,6.28561594418517)\\
 (2,4)&(10.20189235894063,10.20189235894065)&(3.05395019318923,3.05395019318924)\\
 (3,3)&(14.33141430922266,14.33141430922267)&(11.68556492036841,11.68556492036842)\\
 (3,4)&(13.47855498837348,13.47855498837349)&(5.61695538824632,5.61695538824633)
 \end{array}
\end{equation}

In this order, $\eta>10.046,9.428,20.323,19.257$. Derivatives follow by explicit differentiation and algebraic recurrences from $(zF(z))'=\log(1+\e^{-5z}+\cdots+\e^{-20z})$ and the endpoint functions. The critical parameters are enclosed by $180$ rational bisections of their defining cubic polynomials. For the square, the dilogarithm expansion through $110$ terms has remainder at most $|u|^{111}/(111^2(1-|u|))$. For the cube, the Bernoulli expansion of $F$ in Section~\ref{sec:compact}, with its explicit derivative remainder, avoids cancellation at the smaller critical radii. Substitution of the displayed upper bounds into \eqref{eq:local-difference-costs} gives, in the same order,

\begin{equation}\label{eq:local-core-totals}
 \sum C_\bullet<143.316,\quad72.829,\quad119.589,\quad77.191.
\end{equation}

\subsection{The complementary arcs}

For $p=2$, take $t=t_{2,j}$ in the following bounds.

We give bounds for the whole remaining circle.  Near a nontrivial fifth
root, let $1/10\le|s|\le5$.  The error of a shifted rectangular sum is
at most its total variation, whence
\begin{equation}\label{eq:local-variation}
 \begin{split}
 \log|Q_n(\e^{-t/n+\ii\theta})|
 &\le n\Rea F(t-\ii s)+T(t,s),\\
 T(t,s)&=\sum_{h=1}^4\int_0^1
 \frac{|ar_x(\cos\phi_h-r_x)+br_x\sin\phi_h|}
 {(1-r_x)^2+2r_x(1-\cos\phi_h)}\,dx,\\
 a&=5t,\quad b=5s,\quad r_x=\e^{-ax},\quad
 \phi_h=2\pi h/5+bx.
 \end{split}
\end{equation}
For a closed interval $[\ell,u]$ in the positive $s$ axis, let
$\Delta=u-\ell$, let $\mathcal L$ be a positive lower bound for
$F(t)-\Rea F(t-\ii s)$, and let $\mathcal T$ bound $T(t,s)$ from above.
Its contribution to the error multiplied by $n^3$ is at most
\begin{equation}\label{eq:local-primitive-tail}
 40\sqrt{V_+/\pi}\,301^{5/2}\Delta u
 \e^{-2(301\mathcal L-\mathcal T)}.
\end{equation}
All four roots and both sides are included.  Use the intervals
$[k/1000,(k+1)/1000]$, $100\le k<200$, and
$[k/100,(k+1)/100]$, $20\le k<500$.
Positive rectangular upper sums in the $x$ variable, with between
$100$ and $3200$ equal subdivisions, give
$301\mathcal L-\mathcal T>12$ and $602\mathcal L>5/2$ on every cell.
The latter inequality makes the bound valid for every $n\ge301$.

For the far arcs, delete the $5/n$ neighbourhoods of the four nontrivial
roots and the $14/n$ neighbourhood of $1$.  The geometric identity
\begin{equation}\label{eq:local-geometric-sum}
 m_n(q)=\sum_{\substack{k<5n\\5\nmid k}}q^k
 =(1-q^{5n})\left(\frac1{1-q}-\frac1{1-q^5}\right)
\end{equation}
and $\e^{5t/(2n)}<101/100$ imply
\begin{equation}\label{eq:local-far-gap}
 \begin{split}
 \frac{|m_n(q)|}{4n}&\le M,
 \qquad M=\max\left\{\frac{101\pi}{7000},
 \frac1{301}+\frac{101\pi}{10000}\right\},\\
 H(t)&=\frac{1-\e^{-10t}}{10t},\qquad
 g_\infty=F(t)-2\log(1+H(t)+1/301+2M)>0.
 \end{split}
\end{equation}
For angles at distance at most $\pi/5$ from $1$, use the alternative
numerator $q+q^2+q^3+q^4$ in \eqref{eq:local-geometric-sum} and the
$14/n$ exclusion.  Otherwise $|1-q|^{-1}<2$ and the $5/n$ exclusions
bound the second denominator.  The inequality
$|1-r\e^{\ii\theta}|\ge2\sqrt r\,|\sin(\theta/2)|$ yields the two
terms in $M$.  Applying the arithmetic--geometric mean inequality to
the squared factors, and using
$\sum_{5\nmid k}|q|^{2k}/(4n)\le H(t)+1/301$, bounds the far-arc
contribution by
\begin{equation}\label{eq:local-far-tail}
 2\sqrt{4\pi V_+}\,301^{9/2}\e^{-602g_\infty}
\end{equation}
after multiplication by $n^3$.  Here $602g_\infty>9/2$.

On the remaining arc $\theta=s/n$, $|s|\le14$, put $\rho=99/100$.
The inequality $|1-w|\le\rho^{-1/2}|1-\rho w|$, valid for $|w|\le1$,
removes the zero at $w=1$.  With $a=5t$, define
\begin{equation}\label{eq:local-rootone-gap}
 \begin{split}
 C_\rho&=\sum_{k=1}^{80}\frac{\rho^k}{k^2},\qquad
 U_\rho=\Li_2(\rho\e^{-a}),\\
 g_1&=F(t)-\frac{U_\rho^2}{aC_\rho}+2\log\rho,\\
 T_1&=4\left\{\log100+
 \sum_{k=0}^{22}\log\frac{1+\rho\e^{-ak\pi/70}}
 {1-\rho\e^{-ak\pi/70}}\right\}.
 \end{split}
\end{equation}
To see the height bound, the smoothed integral is
\begin{equation}\label{eq:local-smoothed-integral}
 \Rea\frac{\Li_2(\rho\e^{-a+\ii b})-\Li_2(\rho)}{a-\ii b}
 \le\frac{U_\rho^2}{4a\Li_2(\rho)}
 \le\frac{U_\rho^2}{4aC_\rho}.
\end{equation}
The first inequality follows by completing the square in
$1/\sqrt{a^2+b^2}$.  Radial variation is at most $\log100$.
On each angular half-period the variation is at most
$\log((1+r)/(1-r))$, with $r$ evaluated at its initial point.
There are at most $23$ such half-periods, since $|b|\le70<23\pi$;
this proves the formula for $T_1$.  The resulting contribution to the
error multiplied by $n^3$ is at most
\begin{equation}\label{eq:local-rootone-tail}
 5\sqrt{V_+/\pi}\,301^{5/2}\,14^2
 \e^{-2(301g_1-T_1)}.
\end{equation}
Here $301g_1-T_1>15$ and $602g_1>5/2$.

Adding \eqref{eq:local-primitive-tail}, \eqref{eq:local-far-tail}, and
\eqref{eq:local-rootone-tail} gives, respectively,
\begin{equation}\label{eq:local-tail-totals}
 \mathcal T_{2,3}<7.057\cdot10^{-9},\qquad
 \mathcal T_{2,4}<1.071\cdot10^{-6}.
\end{equation}

For $p=3$, the critical radii lie in $[90/1000,95/1000]$ and $[105/1000,110/1000]$, respectively. Use the whole-circle bounds \eqref{compact:primitive-tail}--\eqref{compact:far-tail}. Write $U_\lambda$ for the appropriate bound on $n$ times the ordinary normalized integral over an arc, evaluated at $n_0$. For a primitive-root cell, a root-$1$ cell with $b=5n\theta$, the central root-$1$ arc, and the far arcs, the contributions are

\begin{equation}\label{eq:local-cubic-tail-conversion}
 \begin{array}{c|cc}
 &n^3|\Delta\mathcal A|&n^2|\mathcal A|\\\hline
 a_0\le |s|\le5&5s_+n_0U_\lambda&n_0U_\lambda\\
 1/2\le |b|\le70&b_+n_0U_\lambda&n_0U_\lambda\\
 |b|\le1/2&n_0U_\lambda/2&n_0U_\lambda\\
 \mathrm{far}&2n_0^2U_\lambda&n_0U_\lambda
 \end{array}
\end{equation}

Here $s_+,b_+$ are upper cell endpoints. These multipliers follow from $|\e^{-5\ii\theta}-1|\le5|s|/n$, $|\e^{-5\ii\theta}-1|\le|b|/n$, and the universal bound $2$. The primitive-root grid has $50$, $30$, $90$ cells on $[0.1,0.2]$, $[0.2,0.5]$, $[0.5,5]$; the root-$1$ grid has $95$, $60$ cells on $[0.5,10]$, $[10,70]$. Writing $\ell$ for the phase loss, the stronger inequalities $3n_0\ell>5/2$ hold on the finite cells and central root-$1$ arc, while $3n_0\ell>9/2$ holds for the far gap. Thus the bounds apply to every $n\ge n_0$. Their sums satisfy

\begin{equation}\label{eq:local-cubic-tail-totals}
 \begin{array}{c|cc}
 j&\mathcal T_{3,j}&\mathcal U_{3,j}\\\hline
 3&4.11\cdot10^{-17}&7.761\cdot10^{-17}\\
 4&5.05\cdot10^{-16}&9.842\cdot10^{-16}
 \end{array}
\end{equation}

The entries in the second and third columns are strict upper bounds for the total complementary error multiplied by $n^3$ for the difference and by $n^2$ for a coefficient, respectively. The core and complementary bounds prove \eqref{eq:local-difference-bound}. The root enclosures also give

\begin{equation}\label{eq:local-slope-margins}
 \begin{gathered}
 \mathcal L_{2,3}>1.35000012007429,\quad
 \mathcal L_{2,4}>0.74837836102848,\\
 \mathcal L_{3,3}>1.35896856470153,\quad
 \mathcal L_{3,4}>0.69455459593547,\\
 \mathcal L_{p,j}-C^{\Delta}_{p,j}/301>\delta_p,\qquad
 5w_p+\frac{10+j+5/2}{301}<D_p.
 \end{gathered}
\end{equation}

Proposition~\ref{prop:local-difference} follows.

\subsection{The first correction to a coefficient}

For a single coefficient, use $D=(d+pn^2F'(t))/n$, without $5/2$. Its first nonzero term is

\begin{equation}\label{eq:local-linear-coefficient}
 \begin{aligned}
 &\frac1n\left(H_{p,j}-\frac{gD}{pV}\right),\qquad
 \beta_{p,j}=\frac{pVH_{p,j}}{5g},\\
 H_{p,j}&=B_{p,j}(t)-\frac{\Gamma''_{p,j}(t)}{2pV}
             +\frac{gF'''(t)}{2pV^2}.
 \end{aligned}
\end{equation}

Put $g_2=|\Gamma''_{p,j}(t)|$, $f_3=|F'''(t)|$, $\widehat b_0=|B_{p,j}(t)|$, and $\widehat b_1=|B'_{p,j}(t)|$. For $|D|\le D_0$, the separate core errors multiplied by $n^2$ are bounded by

\begin{equation}\label{eq:local-coefficient-costs}
 \begin{aligned}
 E_1&=G_4\mathfrak m_2/24,\qquad E_2=g_2pM_4\mathfrak m_3/48,\\
 E_3&=\frac{g_2}{4}\left(p^2M_3^2\mathfrak m_4/36+
                  pM_3D_0\mathfrak m_3/3+D_0^2\mathfrak m_2\right),\\
 E_4&=\frac{G_3}{6}(pM_3\mathfrak m_3/6+D_0\mathfrak m_2),\\
 E_5&=\frac g6\left(p^3M_3^3\mathfrak m_5/216+
        p^2M_3^2D_0\mathfrak m_4/12+
        pM_3D_0^2\mathfrak m_3/2+D_0^3\mathfrak m_2\right),\\
 E_6&=gpM_5\mathfrak m_3/120,\\
 E_7&=\frac{gpM_4}{24}(pf_3\mathfrak m_4/6+D_0\mathfrak m_3),\\
 E_8&=\widehat b_0pM_4\mathfrak m_2/24,\\
 E_9&=\frac{\widehat b_0}{2}\left(p^2M_3^2\mathfrak m_3/36+
                      pM_3D_0\mathfrak m_2/3+D_0^2\mathfrak m_1\right),\\
 E_{10}&=\widehat b_1(pM_3\mathfrak m_2/6+D_0\mathfrak m_1),\\
 E_{11}&=b_2\mathfrak m_1/2,\qquad E_{12}=R_*\mathfrak m_0.
 \end{aligned}
\end{equation}

The explicit Taylor remainders used here are

\begin{equation}\label{eq:local-coefficient-taylor}
 \begin{aligned}
 a(s)&=-\Gamma''_{p,j}(t)s^2/2+\varepsilon_a(s),
 &|\varepsilon_a(s)|&\le G_4|s|^4/24,\\
 b(s)&=-gs+\varepsilon_b(s),
 &|\varepsilon_b(s)|&\le G_3|s|^3/6,\\
 \psi_0&=pnF'''(t)s^3/6-Ds,
 &|\psi-\psi_0|&\le pnM_5|s|^5/120,\\
 B_{p,j}(t-\ii s)&=B_{p,j}(t)-\ii B'_{p,j}(t)s+\varepsilon_B(s),
 &|\varepsilon_B(s)|&\le b_2s^2/2.
 \end{aligned}
\end{equation}

In the paired real integrand $\e^q(a\cos\psi-b\sin\psi)$, $E_1,\ldots,E_4$ account for the fourth-order amplitude remainder, the real exponential, the cosine, and the odd cubic amplitude remainder. The next three come from $|\sin\psi-\psi|\le|\psi|^3/6$, the fifth-order phase remainder, and the real exponential multiplying the odd leading term. The next four arise from $B_{p,j}/n$; the last is the folded Euler--Maclaurin remainder. Thus the factors $p/6$ and $p/24$ retain the additional phase and exponential contributions in the cube.

For the Gaussian extensions put $P=n_0\nu_{p,-}$, $W=\sqrt{n_0\nu_{p,+}/\pi}\e^{-Pa_0^2}$ and

\begin{equation}\label{eq:local-coefficient-extension}
 \begin{aligned}
 T_0&=W/(Pa_0),\quad T_2=Wa_0/P+T_0/(2P),\quad T_4=Wa_0^3/P+3T_2/(2P),\\
 E_{13}&=n_0^2\{(g_2/2+gD_0)T_2+gpn_0f_3T_4/6+\widehat b_0T_0/n_0\},\\
 E_{14}&=\begin{cases}2\mathcal T_{2,j},&p=2,\\
                         \mathcal U_{3,j},&p=3.\end{cases}
 \end{aligned}
\end{equation}

The Gaussian bounds decrease for $n\ge n_0$. In the square, the ordinary-coefficient bound follows from the same arc majorants: the ratio to the difference bound is $1/(5s_+)\le2$ on a primitive-root cell, $28/(5\cdot14^2)<2$ on the root-$1$ arc, and $1/(2n_0)<2$ on the far arcs. No lower bound on the sine factor is used. In the cube both bounds were retained separately in \eqref{eq:local-cubic-tail-conversion}. Therefore

\begin{equation}\label{eq:local-coefficient-bound}
 \left|\mathcal A_{p,j,n}(m)-\frac1n
       \left(H_{p,j}-\frac{gD}{pV}\right)\right|
 \le\frac{C_{p,j}(D_0)}{n^2},\qquad
 C_{p,j}(D_0)\ge\sum_{k=1}^{14}E_k.
\end{equation}
\begin{equation}\label{eq:local-coefficient-totals}
 \begin{array}{c|c|cc}
 p&D_0&C_{p,3}(D_0)&C_{p,4}(D_0)\\\hline
 2&11&166&76\\
 2&3&80&34\\
 3&6&29&26
 \end{array}
\end{equation}

The corresponding unrounded sums using the public enclosures are less than $165.084,75.924$; $78.935,33.548$; and $28.404,25.097$, respectively.

\subsection{Integer witnesses and transition indices}

\begin{proposition}\label{prop:local-location}

For $n\ge301$, the nonzero terms in $|m-\alpha_{p,j}n^2|\le w_pn$ change sign exactly once, from positive to negative, and at most one term in the window is zero. If $K_{p,j}^{\rm loc}(n)$ is its last positive index, then

\begin{equation}\label{eq:local-index-bounds}
 |K_{p,j}^{\rm loc}(n)-\alpha_{p,j}n^2-\beta_{p,j}n|<R_{p,j},\qquad
 (R_{2,3},R_{2,4},R_{3,3},R_{3,4})=(61,48,24,40).
\end{equation}
\end{proposition}
\begin{proof}

For the square take $m_- =\lfloor\alpha_{2,j}n^2-2n\rfloor$ and $m_+=\lceil\alpha_{2,j}n^2+2n\rceil$. Then $D\in[-11,-9]$ and $D\in[9,11]$, respectively, and the full coefficient budgets give

\begin{equation}\label{eq:local-endpoint-margins}
 \begin{array}{c|cc}
 j&n\mathcal A_{2,j,n}(m_-)&-n\mathcal A_{2,j,n}(m_+)\\\hline
 3&>1.7402&>2.0228\\
 4&>1.4115&>0.7781
 \end{array}
\end{equation}

For the cube, wide endpoint signs also follow from Proposition~\ref{prop:compact-window}, by taking

\begin{equation}\label{eq:local-cubic-wide-witnesses}
 m_- =\lfloor\alpha_{3,j}n^2-6n\rfloor-1,\qquad
 m_+ =\lceil\alpha_{3,j}n^2+6n\rceil+1.
\end{equation}

These lie strictly outside the window but within $|m-\alpha_{3,j}n^2|\le6n+2$. Their ordinary saddles lie in $(0,1/2)$, since

\begin{equation}\label{eq:local-cubic-interface}
 \begin{aligned}
 \mu_3(1/2)&<\alpha_{3,j}-6/301-2/301^2,\\
 \alpha_{3,3}+6/301+(2+3/5)/301^2&<2.157<3,\\
 \alpha_{3,4}+6/301+(2+4/5)/301^2&<2.005<3.
 \end{aligned}
\end{equation}

The upper interface retains the residue shift. The left endpoint is positive and the right endpoint negative. To place the transition inside the stated window and sharpen its location, we give closer integer witnesses for both powers.

\begin{equation}\label{eq:local-center}
 S_{p,j,n}=\alpha_{p,j}n^2+\beta_{p,j}n-j/5,\qquad
 D_{1,p,j}=5\beta_{p,j},\qquad d_2=3,\quad d_3=6.
\end{equation}

The leading term is $-\mathcal L_{p,j}(m-S_{p,j,n})/n^2$. Put $a_{p,j}=C_{p,j}(d_p)/\mathcal L_{p,j}$ and choose

\begin{equation}\label{eq:local-integer-witnesses}
 m_L=\lceil S_{p,j,n}-a_{p,j}\rceil-1,\qquad
 m_R=\lfloor S_{p,j,n}+a_{p,j}\rfloor+1.
\end{equation}

Even when a threshold is an integer,

\begin{equation}\label{eq:local-witness-distances}
 m_L<S_{p,j,n}-a_{p,j},\quad m_R>S_{p,j,n}+a_{p,j},\quad
 |m_\lambda-S_{p,j,n}|\le a_{p,j}+1\quad(\lambda\in\{L,R\}).
\end{equation}

The following rational bounds suffice:

\begin{equation}\label{eq:local-witness-budget}
 \begin{array}{c|cccc}
 (p,j)&(2,3)&(2,4)&(3,3)&(3,4)\\\hline
 a_{p,j}<&60&46&580/27&2600/69\\
 |\beta_{p,j}|<&0.105&0.424&0.208&0.887
 \end{array}
\end{equation}

Both required domains follow: for $n\ge301$,

\begin{equation}\label{eq:local-witness-domains}
 \begin{aligned}
 |D(m_\lambda)|&=\left|5\beta_{p,j}+5(m_\lambda-S_{p,j,n})/n\right|
 \le5|\beta_{p,j}|+5(a_{p,j}+1)/n<d_p,\\
 |m_\lambda-\alpha_{p,j}n^2|&\le|\beta_{p,j}|n+j/5+a_{p,j}+1<w_pn.
 \end{aligned}
\end{equation}

Thus $\mathcal A_{p,j,n}(m_L)>0$ and $\mathcal A_{p,j,n}(m_R)<0$. Strict monotonicity proves the single transition and the possible single zero in the window. It also gives

\begin{equation}\label{eq:local-integer-rounding}
 m_L\le K_{p,j}^{\rm loc}(n)<m_R,\qquad
 |K_{p,j}^{\rm loc}(n)-S_{p,j,n}|\le a_{p,j}+1.
\end{equation}

Finally,

\begin{equation}\label{eq:local-final-radii}
 \begin{aligned}
 80/\mathcal L_{2,3}+1+3/5&<61,&34/\mathcal L_{2,4}+1+4/5&<48,\\
 29/(27/20)+1+3/5&=3116/135<24,&26/(69/100)+1+4/5&=13621/345<40.
 \end{aligned}
\end{equation}

The shift $-j/5$ proves \eqref{eq:local-index-bounds}. The identification with global indices is completed in Section~\ref{sec:completion}. The assertion about zeros here is confined to the transition window.

\end{proof}

\section{The outer region and the exact boundary}\label{sec:boundary}

For both powers the two changes of sign lie in the central region.  The outer analysis has a common algebraic and modular foundation, but the cubic case has two vanishing limiting amplitudes.  We first develop the positive series needed at the boundary, treat the square, and then resolve this cubic cancellation on the whole coefficient contour.  Put
\begin{equation}\label{eq:boundary-c}
 c=\frac{2\pi^2}{75},\qquad
 (s_0,s_1,s_2,s_3,s_4)=(1,-1,-1,1,1).
\end{equation}

\subsection{Positive series and the square linear boundary}

We use the $q$-Pochhammer notation
\begin{equation}\label{eq:boundary-pochhammer}
 \begin{gathered}
 (a;q)_k=\prod_{\ell=0}^{k-1}(1-aq^\ell),\qquad
 (a;q)_\infty=\prod_{\ell=0}^{\infty}(1-aq^\ell),\qquad (a;q)_0=1,\\
 (a_1,\ldots,a_r;q)_k=\prod_{i=1}^r(a_i;q)_k
 \qquad\bigl(k\in\mathbb Z_{\ge0}\cup\{\infty\}\bigr).
 \end{gathered}
\end{equation}
In the finite product $k$ is a nonnegative integer; the infinite products are interpreted as formal series, or analytically when $|q|<1$.  For $s\in\mathbb Z_{\ge0}$ and a formal variable $X$, set
\begin{equation}\label{eq:boundary-RR}
 R_s(X)=\sum_{k\ge0}\frac{X^{k^2+sk}}{(X;X)_k},\qquad
 G=R_0,\quad H=R_1,\quad K=R_2,\quad L=R_3.
\end{equation}
Here $R_s$ counts partitions whose successive parts differ by at least two and whose smallest part is at least $s+1$.  Splitting according to the presence of the smallest allowed part gives
\begin{equation}\label{eq:boundary-recurrence}
 R_s=R_{s+1}+X^{s+1}R_{s+2},\qquad
 G=H+XK,\quad H=K+X^2L,\quad 0\preceq K\preceq H\preceq G\preceq(1+X)H.
\end{equation}
The relation $A\preceq B$ means that every coefficient of $B-A$ is nonnegative.  Write $Q_\infty(q)=\prod_{k\ge1,\,5\nmid k}(1-q^k)$.  The Rogers--Ramanujan identities and the classical five-dissection, in the form recorded in \cite[(3.3)--(3.4)]{Ma} and \cite[Section 3, (3.46)]{LW}, give
\begin{equation}\label{eq:boundary-dissection}
 G=\frac1{(X,X^4;X^5)_\infty},\quad
 H=\frac1{(X^2,X^3;X^5)_\infty},\quad
 Q_\infty(q)=G(q^5)^2-qG(q^5)H(q^5)-q^2H(q^5)^2.
\end{equation}

\begin{lemma}\label{lem:boundary-series}
The series
\begin{equation}\label{eq:boundary-C}
 \mathscr C_s=\frac{R_{s+1}^2-XR_sR_{s+2}}{1-X}
 =\sum_{r\ge0}X^{2r(r+s+1)}R_{s+2r+1}R_{s+2r+2}
\end{equation}
have nonnegative coefficients.  Moreover,
\begin{equation}\label{eq:boundary-C-bound}
 0\preceq\mathscr C:=\mathscr C_0
 =\frac{GH+H^2-G^2}{1-X}\preceq H^2.
\end{equation}
\end{lemma}
\begin{proof}
Writing $\mathscr D_s=R_{s+1}^2-XR_sR_{s+2}$, two applications of \eqref{eq:boundary-recurrence} give
\[
 \mathscr D_s-(1-X)R_{s+1}R_{s+2}=X^{2s+4}\mathscr D_{s+2}.
\]
Iteration proves \eqref{eq:boundary-C}, since the valuation of the remaining formal series tends to infinity.  To prove the upper bound, take an object counted by the $r$th summand for $s=0$ and prepend the parts $2,4,\ldots,2r$ to both of its rows.  The resulting pair is counted by $H^2$.  Its second row has an initial consecutive even staircase of length exactly $r$, whereas the first row has one of length at least $r$.  Thus the operation is injective, and $r$ is recovered from its image.
\end{proof}

We shall also use a simple insertion inequality.  Suppose that $R(X)$ counts ordered $k$-tuples of these partitions, each row having minimum allowed part one or two, and write $R_u=[X^u]R$.  If $r\ge1$ and $M\ge2r+1$, then
\begin{equation}\label{eq:boundary-insertion}
 R_M\ge k\sum_{u=0}^{r-1}R_u+(k-1)R_r.
\end{equation}
Indeed, to a tuple of total weight $u<r$ insert $M-u$ in any row.  This new part differs from every old part by at least two and is the unique part larger than $M/2$.  For $u=r$, at most one row is excluded, namely a row containing the part $r$ when $M=2r+1$.  Deletion of the unique large part recovers both the source tuple and the chosen row.  The same argument, with $r+1$ in place of $r$, yields
\begin{equation}\label{eq:boundary-insertion-plus}
 R_M\ge k\sum_{u=0}^{r}R_u+(k-1)R_{r+1}\qquad(M\ge2r+3).
\end{equation}
If the inserted part is one, its admissible rows are counted separately below.

\begin{lemma}\label{lem:boundary-linear}
For every $n\ge2$ and $0\le d<10n$,
\begin{equation}\label{eq:boundary-linear-sign}
 s_{d\bmod5}c_d^{(2)}(n)\ge0.
\end{equation}
All these inequalities are strict except at $d=9$ and, when $n=2$, at $d=19$.  In particular, $c_9^{(2)}(n)=0$ for every $n\ge2$.
\end{lemma}
\begin{proof}
Throughout the proof the series on the right are evaluated at $X=q^5$.  Squaring \eqref{eq:boundary-dissection} gives
\begin{equation}\label{eq:boundary-square}
 Q_\infty(q)^2=G^4-2qG^3H-q^2G^2H^2+2q^3GH^3+q^4H^4.
\end{equation}
The first omitted factor of $Q_n$ has degree $5n+1$, and two tail parts have total degree at least $10n+2$.  Consequently,
\begin{equation}\label{eq:boundary-tail}
 Q_n(q)^2\equiv Q_\infty(q)^2
 \left(1+\frac{2X^n(q+q^2+q^3+q^4)}{1-X}\right)\pmod{q^{10n}}.
\end{equation}
Define the formal series $S_j(X)$ and $N_j(X)$ by
\begin{equation}\label{eq:boundary-components}
 \begin{aligned}
 Q_\infty(q)^2&=\sum_{j=0}^4q^jS_j(q^5),\\
 (q+q^2+q^3+q^4)Q_\infty(q)^2&=\sum_{j=0}^4q^jN_j(q^5).
 \end{aligned}
\end{equation}
Thus $S_j$ and $N_j$ exclude the external factor $q^j$.  Formula \eqref{eq:boundary-tail} gives
\begin{equation}\label{eq:boundary-coefficient}
 c_{5m+j}^{(2)}(n)=[X^m]\left(S_j+\frac{2X^nN_j}{1-X}\right),
 \qquad0\le m\le2n-1,
\end{equation}
where
\begin{equation}\label{eq:boundary-N}
\begin{aligned}
 N_1&=G^4+X(H^4+2GH^3-G^2H^2),\\
 N_2&=G^4-2G^3H+X(H^4+2GH^3),\\
 N_3&=G^4-2G^3H-G^2H^2+XH^4,\\
 N_4&=G^4-2G^3H-G^2H^2+2GH^3.
\end{aligned}
\end{equation}

For residue one, put $R=G^3H$.  The recurrences imply
\[
 N_1=G^4+2XH^4-X^3H^2K^2\preceq(1+3X)R.
\]
For $m=n+r$, $0\le r\le n-1$, the coefficient of $X^r$ in $N_1/(1-X)$ is therefore at most $R_r+4\sum_{u<r}R_u$.  Inequality \eqref{eq:boundary-insertion} with four rows makes this strictly smaller than $R_m$ when $r\ge1$.  For $r=0$, at least the three $G$ rows admit a part of weight $n$, so $R_n\ge3>R_0$.  Formula \eqref{eq:boundary-coefficient} now gives strict negativity.  The terms with $m<n$ follow directly from \eqref{eq:boundary-square}.

For residue two, put $T=G^2H^2$ and $B=H^4+2GH^3-G^2HK$.  Direct substitution gives
\begin{equation}\label{eq:boundary-residue-two}
\begin{aligned}
 \frac{N_2}{1-X}&=\frac{XB}{1-X}-G^2\mathscr C,\\
 2T-B&=2XH^3K+XH^2K^2+2X^2H^2K^2+X^2HK^3
       +X(1-X)H^2\mathscr C_1.
\end{aligned}
\end{equation}
After division by $1-X$, the second identity has a nonnegative right-hand side.  Hence $N_2/(1-X)\preceq2XT/(1-X)$, and \eqref{eq:boundary-insertion} gives
\[
 c_{5m+2}^{(2)}(n)\le-T_m+4\sum_{u<r}T_u<0\qquad(m=n+r).
\]
For $r=0$ the sum is empty; the earlier coefficients again follow from \eqref{eq:boundary-square}.

For residue three, write $R=GH^3$, $E=GH^2K$, $Z=GHK^2$, and $Y=KH^3$.  Then
\begin{equation}\label{eq:boundary-residue-three}
\begin{aligned}
 -\frac{N_3}{1-X}
 &=R+\frac{2XE+X^2Z+X^2Y}{1-X}+G^2\mathscr C\\
 &\preceq\frac{2+X+X^2}{1-X}R.
\end{aligned}
\end{equation}
Here $E,Z,Y\preceq R$ and $G^2\mathscr C\preceq G^2H^2=R+XE$.  At index $r\ge1$, the last bound is
\[
 2R_r+3R_{r-1}+4\sum_{u=0}^{r-2}R_u
 <3R_r+4\sum_{u=0}^{r-1}R_u\le R_m.
\]
For $r=0$ and $n\ge2$, all four rows admit a part of weight $n$, so $R_n\ge4>2$.  This proves strict positivity in residue three.

For residue four, put $R=H^4$ and $a_u=[X^u]R$, with $a_u=0$ for $u<0$.  One has
\begin{equation}\label{eq:boundary-residue-four}
 \frac{N_4}{1-X}=-X\left(\frac{GH^2K}{1-X}+GK\mathscr C\right),
 \qquad
 2\left(\frac{GH^2K}{1-X}+GK\mathscr C\right)
 \preceq\frac{4+2X-2X^2}{1-X}R.
\end{equation}
For the comparison, use $GH^2K\preceq(1+X)R$, $GK\mathscr C\preceq GH^2K$, and the nonnegative series $(2-X)/(1-X)$.  If $m=n+1+r\le2n-1$, the possible subtraction is at most
\[
 4\sum_{u=0}^{r-2}a_u+6a_{r-1}+4a_r,
\]
whereas \eqref{eq:boundary-insertion-plus} gives $a_m\ge4\sum_{u=0}^{r}a_u+3a_{r+1}$.  The difference of these two bounds is $3a_{r+1}-2a_{r-1}$.  The product formula for $H$ gives $a_{u+2}\ge a_u$, $a_0=1$, $a_1=0$, $a_2=a_3=4$, and $a_u>4$ for $u\ge4$.  Thus its strict positivity for $r\ge1$ follows in the three cases
\begin{equation}\label{eq:boundary-residue-four-strict}
 3a_{r+1}-2a_{r-1}\ge
 \begin{cases}
 a_0=1,&r=1,\\
 3a_3>0,&r=2,\\
 a_{r-1}>0,&r\ge3.
 \end{cases}
\end{equation}
For $r=0$ the coefficient itself is $a_{n+1}-4$, which vanishes only when $n=2$.  For $m\le n$ no subtraction occurs.  These observations give exactly the stated zero coefficients.

To treat residue zero, first apply the same one-layer tail expansion to $Q_n$ itself.  Write $U=G^2$, $V_0=GH$, $W_0=H^2$.  Its five residue components, restricted to degrees below $10n$, are
\begin{equation}\label{eq:boundary-unsquared}
\begin{aligned}
 A_0&=U-\frac{X^{n+1}(V_0+W_0)}{1-X},&
 A_1&=-V_0+\frac{X^n(U-XW_0)}{1-X},\\
 A_2&=-W_0+\frac{X^{n+1}GK}{1-X},&
 A_3&=A_4=-X^n\mathscr C.
\end{aligned}
\end{equation}
Using $G=H+XK$, the comparison needed for $A_1$ has the explicit nonnegative decomposition
\begin{equation}\label{eq:boundary-unsquared-positive}
\begin{aligned}
 (1+X^2)GH-G^2+XH^2
 &=XH(H-K)+X^2(H^2-K^2)+X^3HK\\
 &\succeq0,
\end{aligned}
\end{equation}
since $H\succeq K$.  The two-row versions of \eqref{eq:boundary-insertion}--\eqref{eq:boundary-insertion-plus}, together with
\[
 V_0+W_0\preceq2U,\qquad
 \frac{U-XW_0}{1-X}\preceq\frac{1+X^2}{1-X}V_0,\qquad
 GK\preceq(1+X)W_0,
\]
show that $A_0$ is strictly positive and the other four components are nonpositive in this range.  In the endpoint case involving insertion of a part one, only the $G$ row is used.  In the convolution for $c_{5m}^{(2)}(n)$, the two residues are either both zero or both nonzero.  Thus every summand is nonnegative, and the terms involving the constant coefficient make the sum strictly positive.  This completes the proof.
\end{proof}

\subsection{A local product formula and the complementary arcs}

We first establish a variance bound used in both outer ranges.
\begin{lemma}\label{lem:boundary-variance}
For every $t>0$,
\begin{equation}\label{eq:boundary-variance}
 0<t^3V(t)<2c.
\end{equation}
Consequently, if $\nu=2n$ and $k=\nu c/t$, then $\nu V(t)t^2<2k$.
\end{lemma}
\begin{proof}
Let $\kappa_2(x)$ be the variance of the distribution on $\{0,1,2,3,4\}$ whose probabilities are proportional to $1,x,x^2,x^3,x^4$.  For $x>0$ this variance is strictly positive.  Differentiation of \eqref{eq:radial}, followed by $w=tu$, gives
\begin{equation}\label{eq:boundary-variance-integral}
 V(t)=25\int_0^1u^2\kappa_2(\e^{-5tu})\,du,\qquad
 t^3V(t)=25\int_0^t w^2\kappa_2(\e^{-5w})\,dw.
\end{equation}
The latter expression is positive and strictly increasing in $t$.  To identify its limit, write
\[
 F(t)=\frac{c+\mathcal A(t)}t,\qquad
 \mathcal A(t)=\frac{\Li_2(\e^{-25t})}{25}
              -\frac{\Li_2(\e^{-5t})}{5}.
\]
The function $\mathcal A$ and its first two derivatives decay exponentially as $t\to\infty$.  Hence
\begin{equation}\label{eq:boundary-variance-limit}
 t^3V(t)=2c+t^2\mathcal A''(t)-2t\mathcal A'(t)+2\mathcal A(t)
 \longrightarrow2c.
\end{equation}
Strict increase to this limit proves \eqref{eq:boundary-variance}, and multiplication by $\nu/t$ proves the last assertion.
\end{proof}

Put $z=t-\ii v$, $\delta=t/n$, $N=n/t$, and $x=\e^{-5t}$.  Let
\begin{equation}\label{eq:boundary-lambda}
 \lambda_h(z)=\e^{e_h(z)}
 =\omega_h\prod_{s=1}^4(1-\zeta^{-hs}\e^{-5z})^{1/2-s/5},
 \qquad(\omega_1,\omega_2,\omega_3,\omega_4)=(-\zeta^2,1,1,-\zeta^3).
\end{equation}
The branches in the product are fixed at $\e^{-5z}=0$.  Pairing $s$ with $5-s$ verifies the equality with the endpoint amplitude defined earlier.  In particular $|\lambda_h(t)|=1$ and $\Gamma_{2,j}(z)=\sum_{h=1}^4\zeta^{-hj}\lambda_h(z)^2$.

The eta transformation \cite[Lemma 6.1, (6.2)]{Ma}, valid for $\Rea z>0$, reads
\begin{equation}\label{eq:boundary-eta}
 Q_\infty(\zeta^h\e^{-z/n})
 =\omega_h\e^{cn/z-z/(6n)}
 \frac{P(\zeta^{-h^{-1}}\e^{-6cn/z})}{P(\e^{-30cn/z})},
 \qquad P(w)=(w;w)_\infty.
\end{equation}
The inverse $h^{-1}$ is taken modulo five.  In this subsection assume $t\ge1/2$, $|v|\le2t$, and $\delta\le1/100$.  Define the upper endpoint correction by
\begin{equation}\label{eq:boundary-endpoint-correction}
 b_h(z)=\frac{5z}{2}\sum_{a=1}^4\mathsf B_2(a/5)
 \frac{\zeta^{ha}\e^{-5z}}{1-\zeta^{ha}\e^{-5z}}
 =E_h(z)+\frac z6.
\end{equation}
For the last equality, pairing $a$ with $5-a$ gives
\begin{equation}\label{eq:boundary-endpoint-identity}
 \sum_{a=1}^4\mathsf B_2(a/5)\frac{\zeta^{ha}}{1-\zeta^{ha}}=\frac1{15}.
\end{equation}
Indeed, the four Bernoulli values are $(1,-11,-11,1)/150$, and $w/(1-w)+w^{-1}/(1-w^{-1})=-1$.  Thus \eqref{eq:boundary-endpoint-correction} has the same sign convention as \eqref{eq:endpoint-functions}.

Applying shifted Euler--Maclaurin summation to the tail quotient, through its second endpoint term, gives the logarithmic identity
\begin{equation}\label{eq:boundary-log-product}
 \begin{aligned}
 \log Q_n(\zeta^h\e^{-z/n})
 &=nF(z)-\frac{z}{6n}+\log\lambda_h(z)
   +\frac{b_h(z)}n+R_{h,n}(z)+\rho_{h,n}^{\eta}(z),\\
 \rho_{h,n}^{\eta}(z)
 &=\log P(\zeta^{-h^{-1}}\e^{-6cn/z})-\log P(\e^{-30cn/z}),\\
 |\rho_{h,n}^{\eta}(z)|&\le2u_*,\qquad
 u_* =\e^{-6cn\Rea(1/z)}.
 \end{aligned}
\end{equation}
Here $R_{h,n}$ contains the higher endpoint terms and the integral remainder in the tail summation, while $\rho_{h,n}^{\eta}$ is the logarithm of the transformed product quotient in \eqref{eq:boundary-eta}.  The endpoint term enters as $b_h/n$, and its bounds, together with those for the tail remainder, are
\begin{equation}\label{eq:boundary-product-remainder}
\begin{aligned}
 |b_h(z)|&\le\frac{2|z|x}{5(1-x)},\\
 |R_{h,n}(z)|&\le\frac{5x}{6n^2}
 \left(\frac{|z|^2}{(1-x)^2}
       +\frac{|z|^3(1+x)}{t(1-x)^3}\right)
 <\frac{4|z|^2x}{n^2}.
\end{aligned}
\end{equation}
Since $x<1/12$, $|z|/t\le\sqrt5$, and $N\ge100$, the sum $b_h/n+R_{h,n}+\rho_{h,n}^{\eta}$ has modulus less than $1/10$.  Exponentiating \eqref{eq:boundary-log-product} and using $|\e^w-1|\le2|w|$ for $|w|\le1/10$, we obtain
\begin{equation}\label{eq:boundary-product}
 Q_n(\zeta^h\e^{-z/n})
 =\e^{nF(z)-z/(6n)}\lambda_h(z)(1+\epsilon_h),\qquad
 |\epsilon_h|\le x\left(\frac{|z|}{n}+\frac{8|z|^2}{n^2}\right)+4u_* .
\end{equation}
In deriving the coefficient one in this bound, we used $4/[5(1-x)]<1$.  This proves the local estimate from \cite[proof of Lemma 6.3, (6.10)--(6.11)]{Ma} on the present parameter range.

On the complement of the four arcs of angular radius $2\delta$, the infinite-product bound \cite[Lemma 8.1, (8.1)]{Ma} gives
\begin{equation}\label{eq:boundary-infinite-tail}
 |Q_\infty(\e^{-\delta+\ii\theta})|
 \le2\e^{c/(4\delta)-\delta/6}\qquad(0<\delta\le1/100).
\end{equation}
This bound is uniform on the entire complementary contour.  For the finite tail quotient $T_n=Q_n/Q_\infty$, elementary geometric summation yields
\[
 \log T_n(\e^{-\delta})\le\frac{21Nx}{25(1-x)},\qquad
 cN-nF(t)\le\frac{Nx}{5(1-x)}.
\]
The first inequality follows from $1/(1-\e^{-5\delta})\le1+1/(5\delta)$.  Their combined cost is $26Nx/[25(1-x)]$.  Subtracting it from $3cN/4$ gives
\begin{equation}\label{eq:boundary-finite-tail}
 |Q_n(\e^{-\delta+\ii\theta})|
 \le2\e^{nF(t)-\delta/6-bN},\qquad
 b=\begin{cases}c/3,&1/2\le t<2,\\37c/50,&t\ge2.\end{cases}
\end{equation}
Indeed, $26x/[25(1-x)]<5c/12$ in the first range and is less than $c/100$ in the second.  This derivation preserves the constant two in \eqref{eq:boundary-infinite-tail}.

For coefficient extraction it is convenient to use the positive normalization
\begin{equation}\label{eq:boundary-normalization}
 \mathfrak B_2(t,n,d)=
 \frac{\e^{2nF(t)+dt/n-t/(3n)}}{\sqrt{4\pi n^3V(t)}},
 \qquad d=-2n^2F'(t).
\end{equation}

\subsection{The square: the unbounded saddle range}

\begin{proposition}\label{prop:boundary-edge}
If $t\ge2$, $n/t\ge100$, and $d=-2n^2F'(t)$ is an integer, then
\begin{equation}\label{eq:boundary-edge-sign}
 s_{d\bmod5}\frac{c_d^{(2)}(n)}{\mathfrak B_2(t,n,d)}>\frac3{10}.
\end{equation}
\end{proposition}
\begin{proof}
The limiting folded amplitudes are
\begin{equation}\label{eq:boundary-gamma-infinity}
 \gamma_{2,j}=\sum_{h=1}^4\zeta^{-hj}\omega_h^2,\qquad
 (\gamma_{2,0},\ldots,\gamma_{2,4})
 =\left(\frac{3+\sqrt5}{2},-(1+\sqrt5),-1,\sqrt5-1,
        \frac{3-\sqrt5}{2}\right).
\end{equation}
They satisfy $s_j\gamma_{2,j}=|\gamma_{2,j}|>3/8$ and $|\gamma_{2,j}|<4$.  Put
\begin{equation}\label{eq:boundary-phase}
 k=2cn/t,\qquad
 \Phi(v)=2n\{F(t-\ii v)-F(t)+\ii vF'(t)\},\qquad
 J=\sqrt{\frac{2nV}{2\pi}}\int_{-2t}^{2t}\e^{\Phi(v)}\,dv.
\end{equation}
The phase estimate \cite[Lemma 7.3, (7.7)]{Ma}, with effective exponent $2n$, implies
\begin{equation}\label{eq:boundary-phase-error}
 |J-1|\le\frac9{4k}
 +\frac{5\e^{-49k/850}+(23/10)\e^{-199k/3200}}{\sqrt k}.
\end{equation}
The corresponding normalized absolute integral is at most $\sqrt{15/14}+5\e^{-49k/850}/\sqrt k$.  The hypotheses of this phase estimate are $t\ge2$ and $k>0$.  Since $k>52$, these two bounds are respectively less than $1/10$ and $11/10$.  For example, the endpoint comparisons follow from
\[
 \e^{1274/425}>20,\qquad \e^{2587/800}>25,\qquad \sqrt{52}>7.
\]
Each exponential comparison follows by truncating the positive exponential series.

Throughout the four arcs, $x<1/20000$ and \eqref{eq:boundary-lambda} gives
\[
 |\lambda_h(z)^2|<\frac{101}{100},\qquad
 |\lambda_h(z)^2-\omega_h^2|<\frac{7x}{4}.
\]
Furthermore, \eqref{eq:boundary-product} gives $|\epsilon_h|<1/700000$, and $|\e^{\ii v/(3n)}-1|\le1/150$.  Using $|(1+\epsilon_h)^2-1|\le3|\epsilon_h|$, the error in replacing the folded amplitude by $\gamma_{2,j}$ is bounded by
\begin{equation}\label{eq:boundary-edge-amplitude}
 \frac4{150}+\frac7{20000}+\frac1{50000}.
\end{equation}
Lemma~\ref{lem:boundary-variance} gives $t^3V(t)<2c<201c/100$.  Together with \eqref{eq:boundary-finite-tail}, this bounds the complementary contour by
\begin{equation}\label{eq:boundary-edge-tail}
 4\sqrt{4\pi(201/100)c}\,N^{3/2}\e^{-37cN/25}
 <11N^{3/2}\e^{-481N/1250}<10^{-10}.
\end{equation}
The last expression decreases for $N\ge100$; the endpoint comparison follows from $\e^{962/25}>11000\cdot10^{10}$.  Combining the bounds gives the positive margin
\begin{equation}\label{eq:boundary-edge-margin}
 \frac38\frac9{10}
 -\frac{11}{10}\left(\frac4{150}+\frac7{20000}+\frac1{50000}\right)
 -10^{-10}
 =\frac{9232789997}{30000000000}>\frac3{10}.
\end{equation}
\end{proof}

\subsection{The square: the compact intermediate range}

\begin{proposition}\label{prop:boundary-middle}
If $n\ge301$, $1/2\le t\le2$, and $d=-2n^2F'(t)$ is an integer with $d\equiv j\pmod5$, then
\begin{equation}\label{eq:boundary-middle-result}
 s_j\Gamma_{2,j}(t)>\frac13,\qquad
 \left|\frac{c_d^{(2)}(n)}{\mathfrak B_2(t,n,d)\Gamma_{2,j}(t)}-1\right|<\frac3{10}.
\end{equation}
\end{proposition}
\begin{proof}
We give explicit bounds for each part of the integral.  Set $\nu=2n$ and
\begin{equation}\label{eq:boundary-middle-amplitude}
 \ell=\exp\frac{8x}{5(1-x)},\quad
 L_1=\frac{8x}{1-x},\quad L_2=\frac{40x}{(1-x)^2},\quad
 g=\min_{0\le j\le4}|\Gamma_{2,j}(t)|.
\end{equation}
On the vertical line through $t$, the modulus and first two derivatives of $\lambda_h^2$ are bounded by $\ell$, $\ell L_1$, and $\ell(L_1^2+L_2)$.  Moreover,
\begin{equation}\label{eq:boundary-H}
 \mathscr H_{2,j}(v)=\e^{\ii v/(3n)}\frac{\Gamma_{2,j}(t-\ii v)}{\Gamma_{2,j}(t)},\qquad
 |\mathscr H_{2,j}(v)|<4,\quad \mathscr H_{2,j}(0)=1,
\end{equation}
and the following are bounds for $|\mathscr H_{2,j}'(0)|$ and $\sup|\mathscr H_{2,j}''(v)|$:
\begin{equation}\label{eq:boundary-h12}
 h_1=\frac{4\ell L_1}{g}+\frac1{3n},\qquad
 h_2=\frac{4\ell}{g}
 \left(L_1^2+L_2+\frac{2L_1}{3n}+\frac1{9n^2}\right).
\end{equation}
The bound in \eqref{eq:boundary-H} follows from $(|\gamma_{2,j}|+4(\ell-1))/|\Gamma_{2,j}(t)|<4$.

Derivative bounds follow directly from the dilogarithmic expression for $F$.  For $r\ge0$ define
\begin{equation}\label{eq:boundary-derivatives}
 D_r(t)=\sum_{a\in\{5,25\}}\sum_{i=0}^r
 \binom ri\frac{(r-i)!a^i\Li_{2-i}(\e^{-at})}{a t^{r-i+1}},\qquad
 M_3=\frac{6c}{t^4}+D_3,\quad M_4=\frac{24c}{t^5}+D_4.
\end{equation}
Termwise absolute values give $|(F-c/z)^{(r)}(t-\ii v)|\le D_r(t)$ and hence $|F'''|\le M_3$, $|F''''|\le M_4$.

On $|v|\le t/4$, put $U=v\sqrt{\nu V}$ and write $\Phi=-U^2/2+R+\ii\Psi$.  Taylor's formula gives
\[
 |R|\le\frac{M_4U^4}{24\nu V^2},\qquad
 |\Psi|\le\frac{M_3|U|^3}{6\sqrt\nu V^{3/2}}.
\]
Set $\eta=M_4(t/4)^2/(24V)$ and $b=1-2\eta>0$.  Pairing $v$ and $-v$, and using that $\mathscr H_{2,j}'(0)$ is purely imaginary, bounds the inner relative error by
\begin{equation}\label{eq:boundary-middle-inner}
 E_{\rm in}=\frac1\nu\left\{
 \frac{M_4}{8V^2}b^{-5/2}+\frac{5M_3^2}{24V^3}b^{-7/2}
 +\frac{h_1M_3}{2V^2}b^{-5/2}+\frac{h_2}{2V}b^{-3/2}\right\}.
\end{equation}
This is the fourth, sixth, fourth, and second moment of the envelope $\e^{-bU^2/2}$, respectively.  The omitted Gaussian tails contribute at most
\begin{equation}\label{eq:boundary-middle-gaussian}
 E_{\rm G}=\frac{2\e^{-U_0^2/2}}{\sqrt{2\pi}U_0},\qquad
 U_0=\frac{t\sqrt{\nu V}}4.
\end{equation}

For $1/2\le t\le1$, use the phase inequalities from the proof of \cite[Lemma 7.2]{Ma}, with $y=v/t$ and $k=2cn/t$:
\[
 \Rea\Phi\le-\frac25k\frac{y^2}{1+y^2}\quad(|y|\le1/2),\qquad
 \Rea\Phi\le-\frac{k}{20}\quad(1/2\le|y|\le2).
\]
Integrating these positive envelopes on $1/4\le|y|\le2$, using $\nu Vt^2<2k<3k$ from Lemma~\ref{lem:boundary-variance} and $|\mathscr H_{2,j}|<4$, gives
\begin{equation}\label{eq:boundary-middle-outer-one}
 E_{\rm out}=\frac{16}{7}
 \left(\frac{28\e^{-2k/85}}{\sqrt k}+4\sqrt k\,\e^{-k/20}\right).
\end{equation}
Only the displayed phase inequalities are needed here; they follow from the derivative bounds for $F-c/z$ independently of a lower threshold for $n$.

For $1\le t\le2$, a sharper estimate follows by writing $F=c/z+(F-c/z)$ and applying the second derivative bound in \eqref{eq:boundary-derivatives}.  With
\[
 \xi=1-\frac{5t^3D_2(t)}{2c}>0,\qquad a=\frac{4\xi k}{17},
\]
one obtains $\Rea\Phi\le-\xi k y^2/(1+y^2)$ on $|y|\le2$.  Since $y^2/(1+y^2)\ge4|y|/17$ for $1/4\le|y|\le2$, one may take
\begin{equation}\label{eq:boundary-middle-outer-two}
 E_{\rm out}=8t\sqrt{\frac{\nu V}{2\pi}}\frac{\e^{-a/4}}a.
\end{equation}
The normalized unweighted absolute integral is at most $b^{-1/2}+E_{\rm out}/4$.  Therefore \eqref{eq:boundary-product} and \eqref{eq:boundary-finite-tail} give the two remaining error terms
\begin{equation}\label{eq:boundary-middle-product-tail}
\begin{aligned}
 \epsilon&=x\left(\frac{\sqrt5}{N}+\frac{40}{N^2}\right)
             +4\e^{-6cN/5},\\
 E_{\rm prod}&=\frac{12\ell\epsilon}{g}
             \left(b^{-1/2}+\frac{E_{\rm out}}4\right),\qquad
 E_{\rm far}=\frac{4\sqrt{4\pi n^3V}}g\e^{-2cN/3}.
\end{aligned}
\end{equation}

For completeness, the numerical margins used in these explicit inequalities are recorded in Table~\ref{tab:boundary-middle}.  On each closed interval $[i/1000,(i+1)/1000]$, the expressions above are bounded at $n=301$, using the exact rational endpoints.  In $D_r$, the terms of orders one and below are logarithms or rational functions, and
\begin{equation}\label{eq:boundary-li2-tail}
 0\le\Li_2(u)-\sum_{k=1}^{60}\frac{u^k}{k^2}
 \le\frac{u^{61}}{61^2(1-u)}\qquad(0\le u<1).
\end{equation}
The amplitude is evaluated through its real trigonometric factorization.  These bounds yield $g>1/3$, $b>0$, and the bounds for $\mathscr H_{2,j}$ above on every interval.  All remaining dependence on $n$ is through decreasing negative powers or terms $n^a\e^{-bn}$; in the latter case $bn>a$ holds already at $n=301$.  Thus the same estimates hold for every larger $n$.
\begin{table}[htbp]
\centering
\begin{tabular}{c r c c}
\hline
$t$ & $\#\text{ intervals}$ & $g$ & $E_{\rm in}+E_{\rm out}+E_{\rm G}+E_{\rm prod}+E_{\rm far}$\\
\hline
$[1/2,1]$ & $500$ & $>0.343005$ & $<0.273978$\\
$[1,2]$ & $1000$ & $>0.379331$ & $<0.060861$\\
\hline
\end{tabular}
\caption{Uniform margins on the compact intermediate range.}\label{tab:boundary-middle}
\end{table}
The sum in the last column is bounded on each interval before the maximum is taken.  For example, on $[1/2,501/1000]$ the five terms are bounded above by
\[
 0.257636,\qquad0.002134,\qquad0.000002770,\qquad
 0.014206,\qquad10^{-21},
\]
whose sum is less than $0.273979$.  The formulas and the interval endpoints specify all estimates in the table; the detailed interval bounds are included in the supplementary material.  Since the total is less than $3/10$, coefficient extraction proves \eqref{eq:boundary-middle-result}.
\end{proof}

\subsection{The square: joining the boundary ranges}

The dilogarithmic formula, or an integration by parts, gives
\begin{equation}\label{eq:boundary-degree-interface}
 -t^2F'(t)<c,\qquad d<2cN^2,\qquad
 2c\cdot100^2<5300.
\end{equation}
Thus a saddle with $t\ge2$ and $d\ge5300$ satisfies the hypotheses of Proposition~\ref{prop:boundary-edge}.  If $d<5300$ and $n\ge530$, Lemma~\ref{lem:boundary-linear} applies.  The only remaining edge degrees lie in the finite set
\begin{equation}\label{eq:boundary-finite-interface}
 1\le n\le529,\qquad
 0\le d\le\min\{5299,\lfloor2n^2/15\rfloor\},
\end{equation}
because $t\ge2$ implies $d<(c/2)n^2<(2/15)n^2$.  This set is included in Lemma~\ref{lem:finite}, whose exact recurrence is independent of the present estimates.

\begin{proposition}\label{prop:boundary-square}
Let $n\ge301$, and suppose the ordinary saddle associated with $d$ satisfies $t\ge1/2$.  Then $s_{d\bmod5}c_d^{(2)}(n)>0$ whenever $d\ge5300$.  If $n\ge530$, the same sign conclusion, allowing the exceptional zero $d=9$, holds for every such degree.  For $301\le n\le529$, the degrees not covered by these assertions are contained in \eqref{eq:boundary-finite-interface}.
\end{proposition}
\begin{proof}
Combine Propositions~\ref{prop:boundary-middle} and \ref{prop:boundary-edge} with Lemma~\ref{lem:boundary-linear} and \eqref{eq:boundary-degree-interface}.
\end{proof}

\subsection{The cube: exact cancellation and the first tail layer}

The cubic boundary has a different origin.  Put

\begin{equation}\label{eq:cubic-infinite}
 S(q)=Q_\infty(q)^3
 =\frac{(q;q)_\infty^3}{(q^5;q^5)_\infty^3}.
\end{equation}

Jacobi's identity, in the same formal-series convention as \eqref{eq:boundary-dissection}, gives

\begin{equation}\label{eq:cubic-jacobi}
 (q;q)_\infty^3=\sum_{\ell\ge0}(-1)^\ell(2\ell+1)
 q^{\ell(\ell+1)/2},\qquad
 [q^{5m+2}]S=[q^{5m+4}]S=0.
\end{equation}

The second assertion follows because triangular numbers have residues only zero, one and three modulo five; it is also the modulus-five case of \cite[Section 3, Lemma 3.1]{LW}.  To compute the finite tail explicitly, we cube \eqref{eq:boundary-dissection} and collect powers modulo five.  This algebraic expansion gives

\begin{equation}\label{eq:cubic-dissection}
\begin{aligned}
 S(q)&=A(X)-qB(X)+q^3D(X),\qquad X=q^5,\\
 A&=G^6-3XGH^5,\qquad B=3G^5H+XH^6,\qquad D=5G^3H^3.
\end{aligned}
\end{equation}

The signs of the first nonzero finite tail layer follow from two positive auxiliary series.

\begin{lemma}\label{lem:cubic-positive-series}

The series

\begin{equation}\label{eq:cubic-UV}
 \mathscr U=\frac{A-B+D}{1-X},\qquad
 \mathscr V=D-\mathscr U
\end{equation}

have strictly positive coefficients in every nonnegative degree.

\end{lemma}\begin{proof}

Increasing the largest part by two, and sending the empty partition to the one-part partition consisting of two, proves $X^2H\preceq H$.  Since $H=K+X^2L$ and $L\preceq K$, we obtain

\begin{equation}\label{eq:cubic-H6}
 H^6=H^5K+X^2H^5L\preceq2H^5K.
\end{equation}

Here the factor $X^2$ is absorbed into one of the five factors $H$; no coefficient comparison between $H$ and $2K$ is required.  Direct expansion, using $(1-X)\mathscr C=GH+H^2-G^2$, yields

\begin{equation}\label{eq:cubic-UV-expansion}
\begin{aligned}
 \mathscr U
 &=4H^6+3XH^5K-(1-2X+X^2)\mathscr C^3,\\
 \mathscr V
 &=H^6+12XH^5K+15X^2H^4K^2+5X^3H^3K^3
   +(1-2X+X^2)\mathscr C^3.
\end{aligned}
\end{equation}

By Lemma~\ref{lem:boundary-series}, $\mathscr C^3\preceq H^6$ and $X^2\mathscr C^3\preceq H^6$.  For the negative term in the second line, use $2X\mathscr C^3\preceq2XH^6\preceq4XH^5K$.  Thus

\begin{equation}\label{eq:cubic-UV-positive}
\begin{aligned}
 \mathscr U&\succeq2H^6+3XH^5K+2X\mathscr C^3,\\
 \mathscr V&\succeq H^6+8XH^5K+\mathscr C^3+X^2\mathscr C^3.
\end{aligned}
\end{equation}

The product for $H$ in \eqref{eq:boundary-dissection} shows that $H^6$ has positive coefficients in degree zero and in every degree at least two.  The series $XH^5K$ supplies degree one.  Both lower bounds are consequently strictly positive in every degree.

\end{proof}

Define the cubic tail quotient by

\begin{equation}\label{eq:cubic-tail-definition}
 T_{3,n}(q)=\prod_{\substack{k>5n\\5\nmid k}}(1-q^k)^{-3},
 \qquad Q_n(q)^3=S(q)T_{3,n}(q).
\end{equation}

Two tail parts have total degree at least $10n+2$, whence

\begin{equation}\label{eq:cubic-first-tail}
 T_{3,n}(q)\equiv1+
 \frac{3X^n(q+q^2+q^3+q^4)}{1-X}\pmod{q^{10n+2}}.
\end{equation}

\begin{lemma}\label{lem:cubic-linear}

For every $n\ge1$,

\begin{equation}\label{eq:cubic-linear}
\begin{aligned}
 c_{5m+2}^{(3)}(n)&=c_{5m+4}^{(3)}(n)=0
       &&(0\le m<n),\\
 c_{5(n+r)+4}^{(3)}(n)&=3[X^r]\mathscr U>0,
 \qquad c_{5(n+r)+2}^{(3)}(n)=-3[X^r]\mathscr V<0
       &&(0\le r<n).
\end{aligned}
\end{equation}

\end{lemma}\begin{proof}

By \eqref{eq:cubic-dissection}, the residue-four and residue-two components of $(q+q^2+q^3+q^4)S(q)$, after removal of their external powers of $q$, are respectively $A-B+D$ and $A-B+XD$.  Dividing these by $1-X$ gives $\mathscr U$ and $-\mathscr V$.  Now apply \eqref{eq:cubic-jacobi}, \eqref{eq:cubic-first-tail}, and Lemma~\ref{lem:cubic-positive-series}.

\end{proof}

\subsection{A modular bound on the whole complementary contour}

The following estimate is adapted to narrower principal arcs than those used for the square.  All angular distances are circular distances.

\begin{lemma}\label{lem:cubic-modular-tail}

Let $N\ge40$, $r=\e^{-1/N}$.  Outside the four arcs $|\theta-2\pi h/5|\le1/N$, $1\le h\le4$, one has

\begin{equation}\label{eq:cubic-modular-tail}
 |Q_\infty(r\e^{\ii\theta})|\le\exp(cN/2+1/100).
\end{equation}

\end{lemma}\begin{proof}

Write $\tau=\theta/(2\pi)+\ii/(2\pi N)$ and $y=\Ima\tau$.  A transformation sending $\tau$ into the standard modular fundamental domain has lower row $(b,-a)$ with $b>0$, $\gcd(a,b)=1$, and

\begin{equation}\label{eq:cubic-fundamental-height}
 Y=\frac{y}{|b\tau-a|^2}\ge\frac{\sqrt3}{2}.
\end{equation}

One can choose such a row by taking a shortest nonzero vector of $\mathbb Z\tau+\mathbb Z$.  It cannot have $b=0$: for some integer $k$, $|\tau-k|<1$ when $N\ge40$.  Put $g=\gcd(b,5)$.  The primitive lower row $(b/g,-5a/g)$ for $5\tau$ gives transformed height $Y_5=g^2Y/5$.  The eta transformation and $Q_\infty(q)=q^{1/6}\eta(\tau)/\eta(5\tau)$ imply \cite[Section 23.18]{DLMF}

\begin{equation}\label{eq:cubic-modular-quotient}
 |Q_\infty|=\e^{-1/(6N)}\sqrt{5/g}\,
 \e^{\pi(Y_5-Y)/12}\frac{|P(q_1)|}{|P(q_5)|},
 \qquad |q_1|=\e^{-2\pi Y},\quad |q_5|=\e^{-2\pi Y_5}.
\end{equation}

Here $P$ is defined in \eqref{eq:boundary-eta}.  The elementary logarithmic bound

\begin{equation}\label{eq:cubic-P-log}
 |\log|P(q)||\le\sum_{k\ge1}\frac{|q|^k}{k(1-|q|^k)}
 \le\frac{\rho}{(1-\rho)^2}\qquad(|q|\le\rho<1).
\end{equation}

If $g=1$, then $|q_1|<1/200$, $|q_5|<17/50$; the sum of the two logarithmic bounds is less than $4/5$.  Since $Y_5-Y<0$, the quotient in \eqref{eq:cubic-modular-quotient} is less than $\sqrt5\e^{4/5}<6$, which is below the right-hand side of \eqref{eq:cubic-modular-tail}.

If $g=5$, write $b=5k$ and $\Delta=\theta-2\pi a/b$.  The two logarithmic errors sum to less than $1/100$, and

\begin{equation}\label{eq:cubic-cusp-height}
 \frac{\pi(Y_5-Y)}{12}
 =\frac{cN}{k^2(1+N^2\Delta^2)}.
\end{equation}

For $k\ge2$ this is at most $cN/4$.  For $k=1$, the rational point is one of the four nontrivial fifth roots and the complementary-arc condition gives $|N\Delta|\ge1$, so it is at most $cN/2$.  These estimates prove the lemma on the whole complementary contour.

\end{proof}

\subsection{The cube: a uniform estimate at the unbounded edge}

Throughout this subsection let $t\ge2$, $N=n/t\ge40$, and $u=\e^{-5t}$.  Introduce

\begin{equation}\label{eq:cubic-scaled-radial}
\begin{aligned}
 G_t(w)&=tF(tw)=\frac c w+\mathscr R_t(w),\\
 \mathscr R_t(w)&=\frac1w\sum_{k\ge1}a_k\e^{-5ktw},\qquad
 a_k=\begin{cases}-1/(5k^2),&5\nmid k,\\4/(5k^2),&5\mid k.\end{cases}
\end{aligned}
\end{equation}

The factor $1/w$ is retained in all derivative estimates.  On $\Rea w=1$ the series and its derivatives give

\begin{equation}\label{eq:cubic-radial-errors}
\begin{aligned}
 u&<\frac1{20000},\qquad
 |\mathscr R_t(w)|\le\frac{u}{5(1-u)},\qquad
 3\{c-G_t(1)\}<u,\\
 \frac{|\mathscr R_t''(w)|}{2}
 &\le\sum_{a\in\{5,25\}}
 \frac{(at)^2+2at+2}{2a}\frac{\e^{-at}}{1-\e^{-at}}
 <\frac1{1000}.
\end{aligned}
\end{equation}

For the first bound, $|a_k|\le1/5$ and $|w|\ge1$ suffice.  Each summand in the derivative bound decreases for $t\ge2$; at $t=2$ it is bounded by

\begin{equation}\label{eq:cubic-radial-rational}
 \frac{122}{10(20000-1)}+
 \frac{2602}{50(20000^5-1)}<\frac1{1000}.
\end{equation}

In particular, with $\overline V=G_t''(1)=t^3V(t)$,

\begin{equation}\label{eq:cubic-scaled-variance}
 2c-\frac1{500}<\overline V<2c+\frac1{500}.
\end{equation}

For residues two and four, \eqref{eq:cubic-jacobi} permits subtraction on the entire Cauchy contour:

\begin{equation}\label{eq:cubic-subtraction}
 c_d^{(3)}(n)=[q^d]\mathscr S_n(q),\qquad
 \mathscr S_n(q)=S(q)\{T_{3,n}(q)-1\}
 \quad(d\equiv2,4\pmod5).
\end{equation}

On $|q|=\e^{-1/N}$, geometric summation and $1/(\e^{1/N}-1)\le N$ give

\begin{equation}\label{eq:cubic-tail-log}
 |\log T_{3,n}(q)|\le\frac{3Nu}{1-u}<4Nu,
 \qquad |\mathscr S_n(q)|\le4Nu\e^{4Nu}|Q_\infty(q)|^3.
\end{equation}

Thus the factor $u$ survives even on the complementary contour.  Define the two positive normalizations

\begin{equation}\label{eq:cubic-normalizations}
\begin{aligned}
 \mathfrak B_3(t,n,d)
 &=\frac{\e^{3nF(t)+dt/n-t/(2n)}}{\sqrt{6\pi n^3V(t)}},\\
 \mathfrak B_3^\sharp(t,n,d)
 &=\frac{\e^{3nF(t)+(d-5n)t/n-t/(2n)}}{\sqrt{6\pi n^3V(t)}}
 =u\mathfrak B_3(t,n,d).
\end{aligned}
\end{equation}

The ordinary and shifted saddle equations are respectively

\begin{equation}\label{eq:cubic-edge-saddles}
 d=-3n^2F'(t),\qquad d-5n=-3n^2F'(t).
\end{equation}

Combining Lemma~\ref{lem:cubic-modular-tail} with \eqref{eq:cubic-tail-log}, \eqref{eq:cubic-radial-errors}, and \eqref{eq:cubic-scaled-variance}, the complementary contour contributes, after division by $\mathfrak B_3^\sharp$, less than

\begin{equation}\label{eq:cubic-complement-budget}
\begin{aligned}
 4\sqrt{6\pi\overline V}\,N^{5/2}
 \e^{-3cN/2+5Nu+3/100+1/(2N)}
 &<\frac{64}{5}N^{5/2}\e^{-1559N/4000+17/400}\\
 &<\frac1{40}.
\end{aligned}
\end{equation}

The last expression decreases for $N\ge40$.  At the endpoint, it suffices to prove $\e^{6219/400}>5242880$, using $\sqrt{40}<32/5$; the Taylor sum through degree thirty already exceeds this integer.  For residues zero, one and three no subtraction is made.  Division by $\mathfrak B_3$ gives the same bound divided by $4N$, and hence a complementary contribution below $1/6400$.

We next fold the four principal arcs.  With the phases $\omega_h$ from \eqref{eq:boundary-lambda}, let

\begin{equation}\label{eq:cubic-folded-amplitude}
 \Lambda_h(x)=\omega_h\prod_{s=1}^4(1-\zeta^{-hs}x)^{1/2-s/5},
 \qquad\gamma_{3,j}(x)=\sum_{h=1}^4\zeta^{-hj}\Lambda_h(x)^3.
\end{equation}

The branches are continued from $x=0$.  On the positive real axis, $\gamma_{3,j}(\e^{-5t})=\Gamma_{3,j}(t)$.  Direct calculation in $\mathbb Q(\zeta)$ gives

\begin{equation}\label{eq:cubic-limiting-amplitudes}
\begin{aligned}
 (\gamma_{3,0}(0),\ldots,\gamma_{3,4}(0))
 &=\left(\frac{5-\sqrt5}{2},-\frac{5+\sqrt5}{2},0,\sqrt5,0\right),\\
 \gamma_{3,2}'(0)&=-\frac32(\sqrt5+1),\qquad
 \gamma_{3,4}'(0)=\frac32(\sqrt5-1).
\end{aligned}
\end{equation}

These identities also follow by expanding the four factors in \eqref{eq:cubic-folded-amplitude} to first order and reducing with $1+\zeta+\zeta^2+\zeta^3+\zeta^4=0$.  For $j=2,4$, write $\gamma_{3,j}(x)=x\mathfrak b_j(x)$.  The absolute sum of the four exponents after cubing is $12/5$, so the coefficient of $x^k$ is bounded in modulus by $4(12/5)_k/k!$.  For $k\ge2$ successive coefficients of this majorant have ratio at most $22/15$.  Consequently,

\begin{equation}\label{eq:cubic-amplitude-majorants}
\begin{aligned}
 |\mathfrak b_j(x)-\mathfrak b_j(0)|&\le17|x|
       &&(j=2,4,\ |x|\le10^{-3}),\\
 |\gamma_{3,j}(x)-\gamma_{3,j}(0)|&<10|x|
       &&(0\le j\le4,\ |x|\le1/20000),\\
 \mathfrak b_4(0)&>9/5,\qquad \mathfrak b_2(0)<-24/5.
\end{aligned}
\end{equation}

For the first line, the coefficient of $x^2$ in the majorant is $408/25$, and $(408/25)/(1-(22/15)10^{-3})<17$.  The second line follows in the same way, starting with coefficient $48/5$.

Put $w=1-\ii v$, $|v|\le1$, and $x_w=\e^{-5tw}$.  The exact eta identity \eqref{eq:boundary-eta}, cubed, gives a multiplicative correction $1+\varepsilon_h^\eta$ satisfying

\begin{equation}\label{eq:cubic-eta-error}
 Q_\infty(\zeta^h\e^{-w/N})^3
 =\omega_h^3\e^{3cN/w-w/(2N)}(1+\varepsilon_h^\eta),
 \qquad |\varepsilon_h^\eta|\le8\e^{-3cN}.
\end{equation}

Indeed, $\Rea(1/w)\ge1/2$, so \eqref{eq:cubic-P-log} bounds both transformed logarithms.  Apply Euler--Maclaurin only on the tail, to $f_a(y)=\log(1-\zeta^{ha}\e^{-5wy})$ at points $t+(k+a/5)/N$, $k\ge0$.  Retain its integral and first endpoint term.  The remaining endpoint terms and the periodic remainder satisfy

\begin{equation}\label{eq:cubic-tail-derivatives}
\begin{aligned}
 \sum_{a=1}^4|\mathsf B_2(a/5)|&=\frac4{25},\qquad
 \sum_{a=1}^4|\mathsf B_3(a/5)|=\frac{18}{125},\qquad
 \sup_{[0,1]}|\mathsf B_3|<\frac1{20},\\
 |f_a'(t)|&\le\frac{5|w|u}{1-u},\qquad
 |f_a''(t)|\le\frac{25|w|^2u}{(1-u)^2},\\
 \int_t^\infty|f_a'''(y)|\,dy
 &\le\frac{25|w|^3u(1+u)}{(1-u)^3}.
\end{aligned}
\end{equation}

The shifted formula \cite[Sections 2.10 and 24.8]{DLMF}, with the three copies of each factor included, therefore yields

\begin{equation}\label{eq:cubic-tail-EM}
\begin{aligned}
 T_{3,n}(\zeta^h\e^{-w/N})
 &=\e^{3N\{G_t(w)-c/w\}}
       \left(\frac{\Lambda_h(x_w)}{\omega_h}\right)^3\e^{\delta_h},\\
 |\delta_h|&\le u\left\{\frac{7}{4N}+\frac{11}{N^2}\right\}.
\end{aligned}
\end{equation}

To see the constants directly, the three contributions, before the harmless simplification using $u<1/20000$, are

\begin{equation}\label{eq:cubic-tail-EM-budget}
 u\left\{\frac{(6/5)\sqrt2}{N(1-u)}
 +\frac{18/5}{N^2(1-u)^2}
 +\frac{5\sqrt2(1+u)}{N^2(1-u)^3}\right\}.
\end{equation}

Both $|\Lambda_h(x_w)|^3$ and $\e^{|\delta_h|}$ are less than $1001/1000$.  Exponentiating and summing the four endpoint errors therefore costs less than $9u/N$.

For the cancelling residues it is essential to keep the eta term in the form $\varepsilon_h^\eta(T_{3,n}-1)$.  Its sum, divided by $\e^{3NG_t(w)-w/(2N)}$, has modulus at most $128uN\e^{-(3c-5u)N}$.  Using the exact cancellation in \eqref{eq:cubic-limiting-amplitudes}, we obtain

\begin{equation}\label{eq:cubic-folded-error}
\begin{aligned}
 \sum_{h=1}^4\zeta^{-hj}\mathscr S_n(\zeta^h\e^{-w/N})
 &=\e^{3NG_t(w)-w/(2N)}x_w\{\mathfrak b_j(0)+\epsilon_j(w)\},\\
 |\epsilon_j(w)|
 &\le17u+\frac9N+128N\e^{-3119N/4000}
 <17u+\frac9N+10^{-8}\qquad(j=2,4).
\end{aligned}
\end{equation}

The factor $x_w$ is now absorbed into the exponential, producing precisely the shifted degree $d-5n$.  For the ordinary residues no subtraction is made and the corresponding folded amplitude is $\gamma_{3,j}(0)+\epsilon_j^o(w)$, where

\begin{equation}\label{eq:cubic-ordinary-error}
 |\epsilon_j^o(w)|\le10u+\frac{9u}{N}+32\e^{-3119N/4000}
 <10u+\frac{9u}{N}+10^{-8}\qquad(j=0,1,3).
\end{equation}

We finish the edge estimate by a direct positive lower bound for the real phase integral.  Put

\begin{equation}\label{eq:cubic-real-kernel}
\begin{aligned}
 \mathscr K_{N,t}(y)
 &=3N\left\{G_t(1-\ii y/\sqrt N)-G_t(1)
              +\frac{\ii yG_t'(1)}{\sqrt N}\right\}
              +\frac{\ii y}{2N^{3/2}},\\
 \mathscr I_{N,t}
 &=\sqrt{\frac{3\overline V}{2\pi}}
       \int_{-\sqrt N}^{\sqrt N}\e^{\mathscr K_{N,t}(y)}\,dy.
\end{aligned}
\end{equation}

\begin{lemma}\label{lem:cubic-real-kernel}

For $t\ge2$ and $N\ge40$, the integral is real and satisfies

\begin{equation}\label{eq:cubic-real-kernel-bounds}
 \mathscr I_{N,t}>\frac45,\qquad
 \sqrt{\frac{3\overline V}{2\pi}}
 \int_{-\sqrt N}^{\sqrt N}|\e^{\mathscr K_{N,t}(y)}|\,dy<\frac32.
\end{equation}

\end{lemma}\begin{proof}

Conjugate symmetry proves reality.  First replace $G_t(w)$ by $c/w$ and write $a=3c\in(39/50,4/5)$.  The reference exponent is

\begin{equation}\label{eq:cubic-reference-kernel}
 -\frac{ay^2}{1+y^2/N}
 +\ii\left\{-\frac{ay^3/\sqrt N}{1+y^2/N}
                  +\frac{y}{2N^{3/2}}\right\}.
\end{equation}

On $|y|\le2$, its imaginary part has modulus less than $6/5$, so $\cos\phi\ge1-\phi^2/2>0$.  Its real exponential can be bounded below by $\e^{-ay^2}$.  The omitted normalized Gaussian tail is less than $3/200$, while the full Gaussian moments of $\phi^2/2$ contribute at most

\begin{equation}\label{eq:cubic-kernel-moments}
 \frac{15}{16aN}+\frac3{8aN^2}+\frac1{16aN^3}<\frac4{125}.
\end{equation}

On $2<|y|\le\sqrt N$, the real part is at most $-ay^2/2$.  Mills bounds give the two tail comparisons

\begin{equation}\label{eq:cubic-kernel-tails}
 \frac{\e^{-4a}}{2\sqrt{\pi a}}<\frac3{200},\qquad
 \frac{\e^{-2a}}{\sqrt{\pi a}}<\frac7{50}.
\end{equation}

By \eqref{eq:cubic-radial-errors}, the difference between the actual and reference exponents has modulus at most $3y^2/1000$.  Integrating the resulting difference of exponentials gives a normalized error below

\begin{equation}\label{eq:cubic-kernel-perturbation}
 \frac{(3/1000)\sqrt a}{2(a/2-3/1000)^{3/2}}<\frac3{500},
 \qquad\sqrt{\frac{\overline V}{2c}}>\frac{499}{500}.
\end{equation}

Combining the lower bounds,

\begin{equation}\label{eq:cubic-kernel-positive-budget}
 \mathscr I_{N,t}>
 \frac{499}{500}\left(1-\frac3{200}-\frac7{50}-\frac4{125}-\frac3{500}\right)
 =\frac{402693}{500000}>\frac45.
\end{equation}

Finally, $\Rea\mathscr K_{N,t}(y)\le-(a/2-3/1000)y^2$.  The square of the normalized absolute Gaussian integral is less than $(a+3/1000)/(a/2-3/1000)<9/4$.  All scalar inequalities above follow at $a\ge39/50$, $N\ge40$ by monotonicity and elementary positive exponential series bounds.

\end{proof}

\begin{proposition}\label{prop:cubic-unbounded-edge}

If $t\ge2$ and $N=n/t\ge40$, then at the shifted saddle in \eqref{eq:cubic-edge-saddles},

\begin{equation}\label{eq:cubic-edge-signs}
 \frac{c_d^{(3)}(n)}{\mathfrak B_3^\sharp}>1\quad(d\equiv4\pmod5),
 \qquad
 \frac{c_d^{(3)}(n)}{\mathfrak B_3^\sharp}<-3\quad(d\equiv2\pmod5).
\end{equation}

At the ordinary saddle,

\begin{equation}\label{eq:cubic-edge-ordinary-signs}
 \frac{c_d^{(3)}(n)}{\mathfrak B_3}>1\quad(d\equiv0,3\pmod5),
 \qquad
 \frac{c_d^{(3)}(n)}{\mathfrak B_3}<-2\quad(d\equiv1\pmod5).
\end{equation}

\end{proposition}\begin{proof}

The change of variable $v=y/\sqrt N$ in the four principal arcs gives the common real integral \eqref{eq:cubic-real-kernel}.  For $j=2,4$, equations \eqref{eq:cubic-folded-error}, \eqref{eq:cubic-complement-budget}, and Lemma~\ref{lem:cubic-real-kernel} imply

\begin{equation}\label{eq:cubic-edge-error-budget}
 \left|\frac{c_d^{(3)}(n)}{\mathfrak B_3^\sharp}
       -\mathfrak b_j(0)\mathscr I_{N,t}\right|
 <\frac32\left(\frac{17}{20000}+\frac9{40}+10^{-8}\right)+\frac1{40}
 =\frac{72755003}{200000000}<\frac38.
\end{equation}

The two signed margins are $(9/5)(4/5)-3/8=213/200>1$ and $-(24/5)(4/5)+3/8=-693/200<-3$.  For the ordinary classes, the error is less than

\begin{equation}\label{eq:cubic-edge-ordinary-budget}
 \frac32\left(\frac{10}{20000}+\frac9{800000}+10^{-8}\right)
 +\frac1{6400}<\frac1{100}.
\end{equation}

The limiting amplitudes in \eqref{eq:cubic-limiting-amplitudes} and the same positive real kernel prove \eqref{eq:cubic-edge-ordinary-signs}.

\end{proof}

\subsection{The cube: the intermediate range with vanishing amplitudes}

On a fixed compact interval there is no need to shift the saddle.  Nevertheless, replacing the two small amplitudes by a common constant bound would destroy the useful scale of their derivatives.  We therefore retain their series in $u=\e^{-5t}$.

\begin{proposition}\label{prop:cubic-middle}

Let $n\ge301$, $1/2\le t\le2$, and $d=-3n^2F'(t)$ be an integer of residue $j$ modulo five.  Then

\begin{equation}\label{eq:cubic-middle-result}
 s_j\Gamma_{3,j}(t)>0,\qquad
 \left|\frac{\mathcal C^{(3)}_{n,d}(t)}{\Gamma_{3,j}(t)}-1\right|
 <0.366338<\frac38.
\end{equation}

\end{proposition}\begin{proof}

Write

\begin{equation}\label{eq:cubic-middle-series}
 \gamma_{3,j}(x)=\sum_{k\ge0}g_{j,k}x^k,
 \qquad |g_{j,k}|\le4\frac{(12/5)_k}{k!}.
\end{equation}

All coefficients are real by conjugate pairing of the four terms.  They are obtained by multiplying four binomial series: for $e_s=3/2-3s/5$, the coefficients of $(1-\zeta^{-hs}x)^{e_s}$ obey

\begin{equation}\label{eq:cubic-binomial-recurrence}
 d_{h,s,0}=1,\qquad
 d_{h,s,k}=\frac{k-1-e_s}{k}\zeta^{-hs}d_{h,s,k-1}\qquad(k\ge1).
\end{equation}

The four products are then weighted by $\zeta^{-hj}\omega_h^3$ and summed.  In particular, $g_{2,0}=g_{4,0}=0$ is imposed as the exact algebraic identity \eqref{eq:cubic-limiting-amplitudes}.  It is never replaced by a nonzero constant error.

For $r=0,1,2$ let

\begin{equation}\label{eq:cubic-middle-derivative-series}
 M_{r,j}(t)=\sum_{k\ge0}|g_{j,k}|(5k)^r u^k,
 \qquad (5k)^0=1.
\end{equation}

These bound the modulus of the corresponding derivatives on the vertical line through $t$.  For $j=2,4$ every sum has a factor $u$.  Truncation at $K=36$ leaves a tail bounded by

\begin{equation}\label{eq:cubic-middle-series-tail}
\begin{aligned}
 \mathscr T_{r,K}(u)
 &=\frac{4(12/5)_{K+1}}{(K+1)!}
       \frac{(5K+5)^r u^{K+1}}{1-q_{r,K}(u)},\\
 q_{r,K}(u)&=u\left(1+\frac7{5(K+2)}\right)
          \left(\frac{K+2}{K+1}\right)^r<1.
\end{aligned}
\end{equation}

This follows from the ratio of successive majorant coefficients.  The same tail with $r=0$, applied with both signs to the truncated signed series, encloses $\Gamma_{3,j}(t)$.  For interval evaluation the upper endpoint of $u$ is used in \eqref{eq:cubic-middle-series-tail}.

For the coefficient integral, use temporarily

\begin{equation}\label{eq:cubic-middle-normalization}
 \widehat{\mathcal C}^{(3)}_{n,d}(t)
 =\e^{t/(2n)}\mathcal C^{(3)}_{n,d}(t)
 =\frac{c_d^{(3)}(n)}{\mathfrak B_3(t,n,d)},\qquad \nu=3n.
\end{equation}

The finite-product leading term contains $\e^{-z/(2n)}$.  Thus on the inner arc $|s|\le t/4$ the normalized amplitude is $\e^{\ii s/(2n)}\Gamma_{3,j}(t-\ii s)/\Gamma_{3,j}(t)$.  Its first two derivative bounds are

\begin{equation}\label{eq:cubic-middle-h12}
 h_1=\frac{M_{1,j}}{|\Gamma_{3,j}(t)|}+\frac1{2n},\qquad
 h_2=\frac{M_{2,j}+M_{1,j}/n+M_{0,j}/(4n^2)}{|\Gamma_{3,j}(t)|}.
\end{equation}

Use $M_3,M_4$ from \eqref{eq:boundary-derivatives}; these bound derivatives of the common radial function, independently of the power.  With $V=V(t)$ and

\begin{equation}\label{eq:cubic-middle-envelope}
 b=1-\frac{M_4(t/4)^2}{12V}>0,\qquad
 U_0=\frac{t\sqrt{3nV}}4,
\end{equation}

the same conjugate-pair argument as in \eqref{eq:boundary-middle-inner} gives

\begin{equation}\label{eq:cubic-middle-core}
\begin{aligned}
 E_{\rm core}&=\frac1{3n}\left\{
 \frac{M_4}{8V^2}b^{-5/2}+\frac{5M_3^2}{24V^3}b^{-7/2}
 +\frac{h_1M_3}{2V^2}b^{-5/2}+\frac{h_2}{2V}b^{-3/2}\right\},\\
 E_{\rm G}&=\frac{2\e^{-U_0^2/2}}{\sqrt{2\pi}U_0}.
\end{aligned}
\end{equation}

For clarity, the phase after putting $U=s\sqrt{3nV}$ is $-U^2/2+R+\ii\Psi$, where $|R|\le M_4U^4/(72nV^2)$ and $|\Psi|\le M_3|U|^3/(6\sqrt{3n}V^{3/2})$.  The four terms in $E_{\rm core}$ are respectively the fourth, sixth, fourth and second Gaussian moments; the purely imaginary linear amplitude term is paired before taking an absolute bound.

To estimate $t/4\le|s|\le t$, put $v=s/t$.  The corrected expansion \eqref{eq:cubic-scaled-radial} gives the exact phase loss

\begin{equation}\label{eq:cubic-middle-phase-loss}
\begin{aligned}
 \mathscr D(t,v)&=G_t(1)-\Rea G_t(1-\ii v)\\
 &=\frac{cv^2}{1+v^2}
 +\sum_{k\ge1}a_ku^k
   \frac{1+v^2-\cos(5ktv)+v\sin(5ktv)}{1+v^2}.
\end{aligned}
\end{equation}

On $1/4\le v\le1$, truncating after $k=24$ leaves an absolute remainder at most

\begin{equation}\label{eq:cubic-middle-phase-tail}
 \frac85\frac{u^{25}}{25^2(1-u)}.
\end{equation}

Indeed, $|a_k|\le4/(5k^2)$ and the remaining trigonometric factor has modulus at most two.  On each closed interval $I=[t_-,t_+]$, divide $[1/4,1]$ into eighty equal closed intervals $J_\ell$.  Let $L_\ell$ be a positive lower bound for $\mathscr D$ on $I\times J_\ell$, obtained from the twenty-four terms and \eqref{eq:cubic-middle-phase-tail}.  If $V_+$ bounds $V$ on $I$, put

\begin{equation}\label{eq:cubic-middle-outer}
 \mathscr O_I=2t_+\sqrt{\frac{3nV_+}{2\pi}}
       \sum_{\ell=1}^{80}\frac3{320}\e^{-3nL_\ell/t_+},
 \qquad E_{{\rm out},j}\le
       \mathscr O_I\frac{M_{0,j}}{|\Gamma_{3,j}(t)|}.
\end{equation}

All quantities on the right are replaced by their uniform bounds on $I$.  In particular, the normalized absolute phase integral on $|s|\le t$ is at most $b^{-1/2}+\mathscr O_I$.

The tail Euler--Maclaurin estimate \eqref{eq:cubic-tail-EM-budget} remains valid here without its final simplification.  Denote that bound by $\delta$ and put $\lambda=\exp\{(12/5)u/(1-u)\}$.  Together with the eta bound in \eqref{eq:cubic-eta-error}, this gives

\begin{equation}\label{eq:cubic-middle-product}
 E_{{\rm prod},j}\le
 \frac{4\lambda\{\e^\delta-1+8\e^{\delta-3cN}\}}
 {|\Gamma_{3,j}(t)|}\left(b^{-1/2}+\mathscr O_I\right).
\end{equation}

The tail error retains its factor $u$.  The eta error need not have this factor on the present compact range, since $N\ge301/2$ and $u\ge\e^{-10}$.

For the rest of the circle, a direct finite-tail estimate is

\begin{equation}\label{eq:cubic-middle-tail-log}
 |\log T_{3,n}(q)|\le\left(\frac{12N}{5}+12\right)\frac{u}{1-u},
 \qquad 3\{c-tF(t)\}\le\frac{3u}{5(1-u)}.
\end{equation}

The first inequality follows by summing the four residue progressions and using $1/(1-\e^{-5/N})\le1+N/5$.  As $N\ge120$, Lemma~\ref{lem:cubic-modular-tail} now gives

\begin{equation}\label{eq:cubic-middle-complement}
\begin{aligned}
 g(t)&=\frac32c-\frac{31}{10}\frac{u}{1-u}>0,\\
 E_{{\rm far},j}&\le
 \frac{\sqrt{6\pi n^3V(t)}}{|\Gamma_{3,j}(t)|}
 \exp\left\{-g(t)N+\frac3{100}+\frac1{2N}\right\}.
\end{aligned}
\end{equation}

It remains to record the uniform arithmetic bounds.  Use the three hundred rational intervals $I=[k/1000,(k+5)/1000]$, $k=500,505,\ldots,1995$, together with the eighty angular intervals above.  Formula \eqref{eq:cubic-binomial-recurrence}, the tails \eqref{eq:cubic-middle-series-tail} and \eqref{eq:cubic-middle-phase-tail}, and the elementary dilogarithm tail \eqref{eq:boundary-li2-tail} supply closed interval bounds for every displayed quantity.  In the radial function it is sufficient to retain thirty-six dilogarithm terms, with tail $u^{37}/(37^2(1-u))$.  The variance can also be evaluated by

\begin{equation}\label{eq:cubic-middle-variance-evaluation}
 Z(t)=\sum_{k=0}^4u^k,\qquad
 F'(t)=\frac{\log Z(t)-F(t)}t,\qquad
 V(t)=\frac{-5\sum_{k=0}^4ku^k/Z(t)-2F'(t)}t.
\end{equation}

On every cell the bounds give $b>0$, the stated signs of the five amplitudes, $3nL_\ell/t_+>1/2$, and $g(t)N>3/2$ already at $n=301$.  Negative powers of $n$ decrease; these last two inequalities also make all exponential envelopes decrease for larger $n$.  Thus the bounds obtained at $n=301$ hold for every $n\ge301$.  Table~\ref{tab:cubic-middle} gives upper bounds for the complete relative error in the temporary normalization.  Each sum is bounded on a single cell before taking its maximum over the cells.

\begin{table}[htbp]
\centering
\begin{tabular}{c c}
\hline
$j$ & $E_{\rm core}+E_{{\rm out},j}+E_{\rm G}+E_{{\rm prod},j}+E_{{\rm far},j}$\\
\hline
$0$ & $<0.105057$\\
$1$ & $<0.095945$\\
$2$ & $<0.346375$\\
$3$ & $<0.100046$\\
$4$ & $<0.363016$\\
\hline
\end{tabular}

\caption{Uniform relative errors for the cube on $1/2\le t\le2$.}\label{tab:cubic-middle}

\end{table}

The largest unrounded bound is below $0.363015557662851$.  Returning from \eqref{eq:cubic-middle-normalization} to the common normalization adds at most $t/(2n)\le1/301$.  Therefore the complete relative error is less than $0.366337816799064<3/8$, as asserted.

\end{proof}

\subsection{The cube: joining the edge for every finite product}

Define

\begin{equation}\label{eq:cubic-sigma}
 \sigma=-3F'(2)=0.197317176514472245349894\ldots,
 \qquad0<\sigma<\frac15.
\end{equation}

The strict upper bound follows from \eqref{eq:boundary-degree-interface}: $\sigma<3c/4=\pi^2/50<1/5$.  The displayed decimal can be enclosed by the elementary dilogarithm series.

\begin{proposition}\label{prop:boundary-cubic}

For every $n\ge1$, the following hold:

\begin{equation}\label{eq:cubic-all-n-edge}
\begin{aligned}
 s_{d\bmod5}c_d^{(3)}(n)&>0
 &&(0\le d\le\sigma n^2,\ d\equiv0,1,3\pmod5),\\
 c_d^{(3)}(n)&=0
 &&(0\le d<5n,\ d\equiv2,4\pmod5),\\
 s_{d\bmod5}c_d^{(3)}(n)&>0
 &&(5n<d\le5n+\sigma n^2,\ d\equiv2,4\pmod5).
\end{aligned}
\end{equation}

The last strict signs also hold throughout $5n<d<10n$, independently of the quadratic cutoff.  The corresponding upper-edge conclusions follow by reciprocity.

\end{proposition}\begin{proof}

First consider $d\equiv2,4\pmod5$.  Degrees below $10n$ are covered by Lemma~\ref{lem:cubic-linear}.  If $10n\le d\le5n+\sigma n^2$, the shifted saddle has $t\ge2$.  The integration-by-parts inequality in \eqref{eq:boundary-degree-interface} gives

\begin{equation}\label{eq:cubic-shifted-interface}
 5n\le d-5n<3cN^2,\qquad
 n\ge256\ \Longrightarrow\ 5n\ge1280>3c\cdot40^2.
\end{equation}

Hence $N>40$ and Proposition~\ref{prop:cubic-unbounded-edge} applies.  For $1\le n\le255$, Lemma~\ref{lem:finite} checks these two residues up to $5n+\lfloor n^2/5\rfloor$, a larger range than required.

For the ordinary classes and $0<d\le\sigma n^2$, the ordinary saddle has $t\ge2$.  If $N\ge40$, apply the same proposition.  Otherwise,

\begin{equation}\label{eq:cubic-stable-interface}
 d<3c\cdot40^2<1265,\qquad d\le1264.
\end{equation}

When $n\ge256$, the first missing factor has degree $5n+1\ge1281$, so these coefficients equal those of $S$.  Their signs are included in the independent Jacobi prefix of Lemma~\ref{lem:finite}.  When $1\le n\le255$, that lemma checks all ordinary residues through $\lfloor n^2/5\rfloor$.  The constant coefficient is one.  Finally, \eqref{eq:reciprocity} supplies the upper edge.  This proves all assertions.

\end{proof}

The interface with Proposition~\ref{prop:cubic-middle} is explicit.  Ordinary saddles with $t\le2$ cover $d\ge\sigma n^2$, whereas the shifted edge for residues two and four extends to $5n+\sigma n^2$.  The two ranges overlap by $5n$ in the actual degree.  Thus the cancellation does not leave an interval between the edge and the ordinary central analysis.

\section{Finite ranges and completion of the proof}\label{sec:completion}

\begin{lemma}\label{lem:finite}

Theorem~\ref{thm:main} holds for \(1\le n\le300\). The additional finite ranges needed in Section~\ref{sec:boundary} have the following signs.

For the square, every degree in

\begin{equation}\label{eq:finite-ranges}
 1\le n\le529,\qquad0\le d\le\min\{5299,\lfloor2n^2/15\rfloor\}
\end{equation}

has the pattern \((+,-,-,+,+)\) according to its residue, except for the zero at \(d=9\), whenever it occurs. For the cube and \(1\le n\le255\),

\begin{equation}\label{eq:cubic-finite-ranges}
\begin{aligned}
 &0\le d\le\lfloor n^2/5\rfloor,\quad d\equiv0,1,3\pmod5:
 &&\sgn c_d^{(3)}(n)=\begin{cases}-1&d\equiv1\pmod5,\\1&d\equiv0,3\pmod5;\end{cases}\\
 &5n\le d\le5n+\lfloor n^2/5\rfloor,\quad d\equiv2,4\pmod5:
 &&\sgn c_d^{(3)}(n)=\begin{cases}-1&d\equiv2\pmod5,\\1&d\equiv4\pmod5.\end{cases}
\end{aligned}
\end{equation}

The initial cubic blocks satisfy \eqref{eq:cubic-zero-blocks}. The infinite cubic product, through degree \(1264\), has, in residue classes \(0,1,3\), the strict signs \((+,-,+)\).

\end{lemma}

\begin{proof}

All the finite statements follow from an integer recurrence, which we specify explicitly. Start with \(a_0=1\), and for \(d>0\) set \(a_d=0\). At stage \(n\), insert each of the four factors with exponents \(k=5n-4,5n-3,5n-2,5n-1\), repeating each insertion \(p\) times. If \(D\) is the current degree, insertion of \(1-q^k\) is the update

\begin{equation}\label{eq:integer-recurrence}
 a_d\longleftarrow a_d-a_{d-k}\qquad(d=D+k,D+k-1,\ldots,k).
\end{equation}

Before the update, coefficients beyond \(D\) are zero. Descending order preserves the old values on the right. Induction over the factors shows that the resulting array at stage \(n\) is exactly the coefficient array of \(Q_n^p\). The same induction applies after truncation modulo \(q^{M+1}\), giving every prefix through degree \(M\).

For each power run through \(n=300\), inspect residue classes \(0,3,4\). Every coefficient in the first class is positive. In each of the other two classes, deleting zeros gives one sign change, with positive initial sign and negative final sign. The numbers of inspected positions are

\begin{equation}\label{eq:finite-count}
 \sum_{n=1}^{300}(6pn^2+1)
 =\begin{cases}108540900,&p=2,\\162811200,&p=3.\end{cases}
\end{equation}

For the square prefix, take \(M=5299\) and run through \(n=529\). Inspecting precisely \eqref{eq:finite-ranges} gives

\begin{equation}\label{eq:prefix-count}
 \sum_{n=1}^{529}\bigl(1+\min\{5299,\lfloor2n^2/15\rfloor\}\bigr)=2102022
\end{equation}

positions, with the stated zero as the only exception to the sign pattern. For the cubic edge, the two lines of \eqref{eq:cubic-finite-ranges} give respectively \(667356\) and \(444733\) positions. The initial zero blocks contribute a further \(65280\) positions. The first nonzero degenerate terms are

\begin{equation}\label{eq:cubic-first-edge}
 c_{5n+2}^{(3)}(n)=-6,\qquad c_{5n+4}^{(3)}(n)=9.
\end{equation}

For the infinite cubic prefix, take \(M=1264\) and multiply all factors with \(5\nmid k\) and \(k\le1264\), three times each. Factors of larger degree cannot affect this prefix. All the claimed signs result directly from integer comparisons. The ancillary transition tables give every last positive index and the exact neighboring transition coefficients; the prefix and edge tables record the remaining finite ranges.

The recurrence can also be evaluated using half the coefficients: multiplication by \(1-q^k\) reverses reciprocal symmetry, so the missing half of every intermediate product is determined before the next update. Independent full-polynomial multiplication, using ordinary convolution for the square, gives the same transition coefficients and signs for \(n\le20\) in the square case and \(n\le30\) in the cubic case.

\end{proof}

\begin{proof}[Proof of Theorems~\ref{thm:main} and \ref{thm:location}]

Lemma~\ref{lem:finite} settles \(n\le300\). Fix \(n\ge301\), \(p\in\{2,3\}\), and a degree \(0<d\le5pn^2\), and choose its ordinary saddle by \eqref{eq:ordinary-saddle}. Write

\begin{equation}\label{eq:global-windows}
 w_2=3,\qquad w_3=6,\qquad
 \mathcal W_{p,j,n}=\{m\in\mathbb Z:|m-\alpha_{p,j}n^2|\le w_pn\}.
\end{equation}

If \(t\ge1/2\), Propositions~\ref{prop:boundary-square}, \ref{prop:boundary-cubic}, and \ref{prop:cubic-middle}, together with Lemma~\ref{lem:finite}, give positive signs in residue classes zero and three, negative signs in class one, and the asserted signs or initial zeros in classes two and four. For the square the only zero in this range is \(d=9\); for the cube they are exactly the initial blocks in \eqref{eq:cubic-zero-blocks}.

All these degrees precede the transition windows. Indeed, the radial estimate \(-t^2F'(t)<c<4/15\) and the enclosures of the critical constants give

\begin{equation}\label{eq:outer-before-window}
 \begin{aligned}
 &\mu_2(1/2)<\frac12,\qquad 1<\alpha_{2,j}<\frac65,\\
 &\mu_3(1/2)<\frac34,\qquad \frac{19}{10}<\alpha_{3,j}<\frac{11}{5},\\
 &n^2\mu_p(1/2)-\frac j5<\alpha_{p,j}n^2-w_pn.
 \end{aligned}
\end{equation}

If \(0\le t\le1/2\), Proposition~\ref{prop:compact-window} gives strict positivity in residue class zero, strict negativity in classes one and two, and the prescribed signs in classes three and four outside \(\mathcal W_{p,j,n}\). The windows lie inside the lower half, since

\begin{equation}\label{eq:windows-interior}
 \alpha_{p,j}-\frac{w_p}{n}>0,\qquad
 \alpha_{p,j}+\frac{w_p}{n}+\frac{j}{5n^2}<p.
\end{equation}

Within each window, Proposition~\ref{prop:local-difference} makes the positively weighted coefficient sequence strictly decreasing. Proposition~\ref{prop:local-location} supplies a positive term and a later negative term in that same window. Therefore, after zeros are omitted, there is exactly one positive-to-negative transition in each window. Strict decrease also permits at most one zero within a window. The separate boundary zeros create no additional sign changes.

This covers the lower half. Under \eqref{eq:reciprocity}, an upper-half coefficient in class zero reflects to a positive coefficient in the same class. An upper-half coefficient in class three or four reflects to a lower-half coefficient in class two or one. Such reflected coefficients are negative, apart from the terminal cubic zero block in class three. The endpoint values are

\begin{equation}\label{eq:sequence-ends}
 \begin{aligned}
 &c_0^{(p)}(n)=c_{10pn^2}^{(p)}(n)=1,\\
 &c_3^{(2)}(n)=2,\quad c_4^{(2)}(n)=1,\quad
 c_{20n^2-2}^{(2)}(n)=-1,\quad c_{20n^2-1}^{(2)}(n)=-2,\\
 &c_3^{(3)}(n)=5,\quad c_{5n+4}^{(3)}(n)=9,\quad
 c_{30n^2-5n-2}^{(3)}(n)=-6,\quad c_{30n^2-1}^{(3)}(n)=-3.
 \end{aligned}
\end{equation}

Together with the zero blocks, these observations prove Theorem~\ref{thm:main}.

The local last positive index in Proposition~\ref{prop:local-location} is now the last positive index of the whole sequence. Its four bounds are precisely \eqref{eq:main-location}. Dividing by \(n^2\) and letting \(n\to\infty\) proves \eqref{eq:main-limits}, and hence Theorem~\ref{thm:location}.

\end{proof}

\paragraph{Supplementary material.}

The ancillary files contain the exact transition, prefix, and edge tables, the complete interval bounds underlying the displayed estimates, and programs reproducing these quantities. Parameter intervals have rational endpoints; transcendental expressions are enclosed using interval operations and the explicit remainders stated above. The critical roots are enclosed by rational bisection. The analytic bounds were reevaluated at \(70\) decimal digits for the square and \(100\) digits for the cube; the integer recurrences use arbitrary-precision integers. The records give parameter coverage and intermediate bounds, allowing each finite inequality used in the proof to be located and reproduced.

\end{document}